\documentclass{amsart}

\usepackage{amsmath, amssymb, amsthm, amsfonts}
\usepackage{adjustbox}
\usepackage{orcidlink}
\usepackage{mathtools}
\usepackage{geometry}
\usepackage{hyperref}
\definecolor{vegasgold}{rgb}{0.77, 0.7, 0.35}
\definecolor{darkgoldenrod}{rgb}{0.72, 0.53, 0.04}
\definecolor{gold(metallic)}{rgb}{0.83, 0.69, 0.22}
\hypersetup{
 colorlinks=true,
 linkcolor=darkgoldenrod,
 filecolor=brown,      
 urlcolor=gold(metallic),
 citecolor=darkgoldenrod,
 }

\newtheorem{lthm}{Theorem}

\usepackage{graphicx}
\usepackage{tikz-cd}
\usepackage{enumitem}

\newcommand{\Q}{\mathbb{Q}}
\newcommand{\C}{\mathbb{C}}
\newcommand{\op}[1]{\operatorname{#1}}
\newcommand{\R}{\mathbb{R}}
\newcommand{\Z}{\mathbb{Z}}
\newcommand{\GL}{\mathrm{GL}}

\newcommand{\cO}{\mathcal{O}}

\newcommand{\Tr}{\operatorname{Tr}}

\theoremstyle{plain}
\newtheorem{theorem}{Theorem}[section]
\newtheorem{lemma}[theorem]{Lemma}
\newtheorem{proposition}[theorem]{Proposition}

\theoremstyle{definition}
\newtheorem{definition}[theorem]{Definition}

\theoremstyle{remark}
\newtheorem{remark}[theorem]{Remark}

\numberwithin{equation}{section}

\title[Shapes of octic Kummer extensions]
{On the distribution of shapes of octic Kummer extensions}

\author[A.~Jakhar]{Anuj Jakhar\, \orcidlink{0009-0007-5951-2261}}
\address{Indian Institute of Technology Madras, Chennai, India}
\email{anujjakhar@iitm.ac.in}

\author[A.~Ray]{Anwesh Ray\, \orcidlink{0000-0001-6946-1559}}
\address{Chennai Mathematical Institute, Chennai, India}
\email{anwesh@cmi.ac.in}


\begin{document}

\maketitle

\begin{abstract}
The \emph{shape} of a number field $K$ of degree $n$ is defined as the
equivalence class of the lattice of integers under linear operations
generated by rotations, reflections, and positive scalar dilations.
It may be viewed as a point in the space of shapes
\[
\mathcal{S}_{n-1}
=
\GL_{n-1}(\Z)\backslash \GL_{n-1}(\R)/\op{GO}_{n-1}(\R).
\]
In this paper, we study the distribution of shapes of octic Kummer
extensions $L=\Q(i,\sqrt[4]{m})$, where $m\in\Z[i]$ is fourth-power-free.
We parametrize these shapes by explicit invariants known as shape parameters and
establish an asymptotic formula for their joint distribution ordered by
absolute discriminant.
The limiting distribution is given by an explicit measure that factors
as the product of a continuous measure and a discrete
measure arising from local arithmetic conditions.
\end{abstract}

\section{Introduction}

\subsection{Motivation and historical context}
The interplay between the arithmetic properties of a number field and the geometric structure of its ring of integers is a central theme in algebraic number theory, going back to the foundational works of Dirichlet, Hermite, and Minkowski. Let $K$ be a number field of degree $n = [K : \Q]$. The Minkowski embedding
\[
J: K \hookrightarrow K_\R := K \otimes_\Q \R \cong \R^{r_1} \times \C^{r_2} \cong \R^n
\]
identifies the ring of integers $\cO_K$ with a lattice $\Lambda_K=J(\cO_K)$ of full rank in the $n$-dimensional real vector space $K_\R$. The classical study of these lattices has been largely governed by a single invariant, namely, the discriminant $\Delta_K$. Geometrically, the square root of the absolute discriminant corresponds to the covolume of the lattice $\Lambda_K$. Although the discriminant is a powerful metric which encodes the ramification in number fields, it is a scalar quantity that ignores the finer geometric properties of the lattice. Indeed, two lattices may have the same covolume, however the set of angles between generating vectors may be different. To capture this finer geometric information, we consider a more refined invariant, called the \emph{shape} of the number field.

To define the shape rigorously, we view the Minkowski space $K_\R$ as a Euclidean space equipped with the inner product defined by \[\langle x, y \rangle:=\sum_{\tau:K\hookrightarrow \mathbb{C}} \tau(x)\bar{\tau} (y).\] Since $\Lambda_K$ contains the distinguished vector $J(1) = (1, \dots, 1)$, this imposes a rigid linear constraint common to all number fields. It is thus standard practice to project the lattice onto the trace-zero hyperplane
\[
H := \{ v \in K_\R \mid \langle v, J(1) \rangle = 0 \}.
\]
The \emph{shape} of $K$ is defined as the equivalence class of this projected lattice of rank $n-1$, considered modulo the group $\op{GO}_{n-1}$ of orthogonal similitudes (the group generated by rotations, reflections, and uniform scaling). The space of all such shapes, denoted by $\mathcal S_{n-1}$, is the double coset space
\[
\mathcal S_{n-1} := \GL_{n-1}(\Z) \backslash \GL_{n-1}(\R) / \op{GO}_{n-1}(\R).
\]
This space has a natural measure $\mu$ derived from the Haar measure on $\GL_{n-1}(\R)$. When $n=2$, the space of shapes can be identified with the modular surface $\mathbb{H}/\op{GL}_2(\mathbb{Z})$, where $\mathbb{H}$ is the complex upper half plane and $\op{GL}_2(\mathbb{Z})$ acts on $\mathbb{H}$ by fractional linear transformations. Under this identification, $\mu$ on $\mathcal{S}_2$ corresponds to the hyperbolic measure $\frac{dx\, dy}{y^2}$.

A central question in arithmetic statistics is how these shapes are distributed as $K$ varies over a family of number fields with fixed degree ordered by absolute discriminant. A number field $K/\mathbb{Q}$ of degree $n$ with Galois closure $\widetilde{K}$ is called \emph{generic} if $\op{Gal}(\widetilde{K}/\Q)\simeq S_n$. For generic families of number fields with degree $n$, it is expected that the shapes are equidistributed in $\mathcal S_{n-1}$ with respect to the natural measure $\mu$. This has been verified in low degrees: Terr \cite{Terr97} proved that the shapes of cubic fields are equidistributed in $\mathcal S_2$, and Bhargava and Harron \cite{BH16} subsequently extended this result to $S_4$-quartic and $S_5$-quintic fields.

\par The situation changes significantly when one restricts attention to sparse families of number fields $K$. Often such families have \emph{non-generic} Galois groups, i.e., $\op{Gal}(\widetilde{K}/K)$ is a proper subgroup of $S_n$. In such thin families, arithmetic constraints force the shapes to lie in lower-dimensional subsets of $\mathcal S_{n-1}$. The first major result was discovered by Harron \cite{Har17} in his study of pure cubic fields $K = \Q(\sqrt[3]{m})$. Although these fields are generic, they still form a thin family of fields in the entire family of cubic fields. Harron showed that the shapes of these fields do not fill the space $\mathcal S_2$ but instead lie on two specific one-dimensional geodesics in the modular surface. This leads to a tame-wild dichotomy, which is to say that the shape lies on one geodesic when the prime $3$ is tamely ramified, and on a different geodesic when $3$ is wildly ramified. Since these geodesics have measure zero in the ambient space, the resulting distribution is supported entirely on these one-dimensional sets. This phenomenon was later extended to pure fields $\Q(\sqrt[p]{m})$ by Holmes \cite{Holmes}, where $p$ is a fixed prime. The shapes of such pure fields lie on subspaces of dimension $(p-1)/2$. Harron \cite{HarronANT} also established distribution results for complex cubic fields with a fixed quadratic resolvent. In this setting, the shape is constrained to lie on a specific geodesic within the modular surface, determined by the trace-zero form of the field. 

\par More recently, investigations of pure quartic fields \cite{purequartic} and pure sextic fields \cite{JKRR26} have revealed that, in composite degree, the space of shapes decomposes into a finite union of translated torus orbits. Since these orbits are lower-dimensional, the
restriction of the ambient Haar-induced measure $\mu$ to any one of
them is the zero measure. Each orbit nevertheless carries its own
natural invariant measure, obtained from Haar measure on the torus
whose orbit is being considered. However, the measure governing the distribution of shapes on each such orbit is not given in terms of this torus invariant measure, but rather by a distinct measure reflecting local arithmetic constraints. In these cases, the limiting distribution is described by a product of a continuous measure (governed by Archimedean valuations of the generator) and a discrete measure (governed by local congruence conditions). These works collectively highlight that for non-generic families, the shape acts as a precise geometric mirror reflecting the algebraic structure of the field.

\par Mantilla-Soler and Monsurró \cite{MSM16} determined the shapes of cyclic
number fields of prime degree~$\ell$, while Harron and Harron
\cite{harronharrongaloisquartic} studied shape distribution questions for
Galois quartic number fields. To the best of our knowledge, the shapes of families of number fields with
nonabelian Galois group have not been systematically investigated.
Kummer extensions of fixed degree form a natural but sparse family of such
Galois extensions.
In this paper, we study the distribution of shapes of \emph{octic Kummer
extensions} $L=\mathbb{Q}(i,\sqrt[4]{m})$, where $m\in\mathbb{Z}[i]$ is
fourth-power-free. Since $\mathcal{O}_L$ contains $\Z[i]$, the Minkowski lattice \( \Lambda_L \) contains not only the vector \( J(1) \) but also \( J(i) \). This does impose some constraints on the shape matrices and in order to resolve this, we introduce a relative notion of shape by considering the orthogonal projection of the lattice $\Lambda_L$ onto the orthogonal complement of the plane spanned by both $J(1)$ and $J(i)$. The resulting projected lattice has rank 6, and we define the shape of the octic Kummer extension to be the equivalence class of this lattice in the shape space $\mathcal S_6$.

\par A key part of the analysis is the explicit determination of integral bases, which involves subtle congruence conditions in the Gaussian integers. Using the recent classification of Venkataraman and Kulkarni \cite{quartickummer}, we compute explicit Gram matrices for the projected lattices in terms of arithmetic parameters derived from the prime factorization of $m$. Writing $m = fg^2h^3$ for squarefree, coprime Gaussian integers $f, g, h$, we identify two fundamental real invariants called \emph{shape parameters}:
\[
\lambda_1 = \frac{|f|}{|h|} \quad \text{and} \quad \lambda_2 = |g|,
\]
where $|\cdot|$ denotes the complex norm. These invariants play the same role as the shape parameters in the pure prime degree case. Our main result establishes that as the discriminant grows, the shapes of octic Kummer fields become equidistributed on specific loci defined by $(\lambda_1, \lambda_2)$ and an asymptotic formula for the limiting distribution is established.

\par The present paper is closely related to earlier work on pure
quartic fields \cite{purequartic} and pure sextic fields \cite{JKRR26}.
The general strategy is similar in all three settings.  One begins with
an explicit description of the integral bases, uses these bases to
compute Gram matrices, isolates the finitely many possibilities arising
from local congruence conditions, and then expresses the resulting
shapes in terms of a small number of arithmetic parameters.  The
distribution of these parameters is subsequently studied using
geometry-of-numbers and sieve arguments.

There is, however, an important difference in the present setting.  Since
$L$ contains $K=\Q(i)$, its ring of integers contains
$\cO_K=\Z[i]$.  Thus the Minkowski lattice associated with
$\cO_L$ contains a distinguished rank-two sublattice coming from
$\Z[i]$.  In defining the shape relevant to this family, we remove
this fixed contribution by projecting onto the orthogonal complement of
the real span of $J(\Z[i])$.  The resulting lattice has rank $6$, and
it is this lattice whose shape we study.  When $m\in\Z$, the field
$\Q(\sqrt[4]{m})$ is a quartic subfield of
$L=\Q(i,\sqrt[4]{m})$, so the present construction is naturally
related to the pure quartic setting considered in
\cite{purequartic}. For further details, see Definition \ref{def:relative-shape} and Remark \ref{rem:pure-quartic-relation}.

\subsection{Main Result} 
Let us briefly describe the distribution result for the pair of shape parameters introduced above. The absolute discriminant of $L$ is given, up to a constant $C>0$ depending
only on congruence conditions modulo a certain power of $(1+i)$, by $|\Delta_L|= C|f g^2 h^3|^6$.
Accordingly, we order triples $(f,g,h)$ of elements in $\Z[i]$ that are squarefree and mutually coprime according to their \emph{height} $H(f,g,h):=|f|\,|g|^2\,|h|^3$. Within each congruence class, this is indeed equivalent to ordering by discriminant.
For a fixed rectangle
\[
\mathcal R=[R_1',R_1]\times[R_2',R_2]\subset(0,\infty)^2,
\]
we study the distribution of pairs $(\lambda_1,\lambda_2)\in\mathcal R$
as $H(f,g,h)\le X$ and $X\to\infty$, under the additional assumption
that $(f,g,h)$ is \emph{strongly carefree}, i.e., $f,g,h$ are squarefree
and pairwise coprime in $\Z[i]$. Let $N(\mathcal R;X)$ denote the number of strongly carefree triples
$(f,g,h)$ with height at most $X$ and shape parameters in $\mathcal R$. We define a discrete measure $\mu_{\mathrm{sf}}$
on $\R_{>0}$ by assigning mass
\[
\alpha(g)
=
\frac{1}{|g|^{2}}
\prod_{\mathfrak p\mid g}
\left(
\frac{\op{Norm}(\mathfrak p)}
{\op{Norm}(\mathfrak p)+2}
\right)
\]
to each squarefree $g\in\Z[i]$, and zero otherwise.
We also define a continuous measure $\mu_\infty$ on $\R_{>0}$ by
\[
d\mu_\infty(x)
=
\pi^2
\prod_{\mathfrak p}
\left(1-\dfrac{3}{\op{Norm}(\mathfrak p)^2}+\dfrac{2}{\op{Norm}(\mathfrak p)^3}\right)\,dx.
\]
Their product yields a measure $\hat{\mu}:=\mu_\infty\times \mu_{\mathrm{sf}}$ on $\mathcal R$. Our main result identifies the limiting distribution of shape
parameters with respect to this measure.

\begin{lthm}\label{thm a}
With notation as above, one has
\[
\lim_{X\to\infty}
\frac{N(\mathcal R;X)}{X}
=
\int_{\mathcal R}\hat{\mu}.
\]
\end{lthm}
This provides a concrete
description of the distribution of shapes in this non-generic family of octic number fields and illustrates how arithmetic constraints influence the resulting
geometry of integer lattices.

\subsection{Organization}
Including the introduction, the article consists of four sections. In \S2, we recall the notion of the shape of a lattice of rank $n$ and its realization as a point in the space of shapes $\mathcal{S}_n$. We review equivalent descriptions using change-of-basis matrices and Gram matrices, and explain how the natural invariant measure on the space of shapes arises from Haar measure on $\GL_m(\R)$. We then specialize this framework to number fields, defining the shape of a number field via the Minkowski embedding and the projection of the ring of integers onto the trace-zero subspace. Finally, we review the explicit description of integral bases for the octic Kummer extensions $L=\Q(i,\sqrt[4]{m})$, following the classification of \cite{quartickummer}. These bases, which depend on congruence conditions on $m$, provide the arithmetic input needed for shape computations. In \S3, we compute the Gram matrices associated to the
Minkowski lattices of the rings of integers of the octic Kummer extensions
$L=\Q(i,\sqrt[4]{m})$. We show that the lattice shapes are governed by the parameters $(\lambda_1, \lambda_2)$ in an asymptotic sense as $|m|\rightarrow \infty$. In Section~\S4, we study the distribution of the shape
parameters $(\lambda_1,\lambda_2)$ by counting triples $(f,g,h)\in\Z[i]^3$
ordered by a natural height arising from the discriminant.
We first obtain an asymptotic formula without arithmetic restrictions using
geometry-of-numbers arguments. Then we impose squarefreeness and
coprimality conditions and prove Theorem \ref{thm a}.

\subsection*{Declarations}
\begin{description}
\item[Acknowledgements]
The authors thank the referee for a careful reading of the manuscript and for helpful comments and suggestions that improved the exposition.
\item[Conflict of interest] Not applicable, there is no conflict of interest to report.
    \item[Availability of data and materials] No data was generated or analyzed in obtaining the results in this article.
\end{description}

\section{Preliminary notions} 

\subsection{Shapes of lattices}

Let $V$ be a real vector space of finite dimension $m$, equipped with an inner product $\langle \cdot,\cdot\rangle$. Fix an orthonormal basis $e_1,\dots,e_m$ of $V$, so that $\langle e_k,e_j\rangle=\delta_{k,j}$. Let $\Lambda\subset V$ be a lattice of full rank. The \emph{shape} of $\Lambda$ is an invariant designed to capture the geometry of $\Lambda$ up to overall scaling and orthogonal change of coordinates.

To make this precise, let $\op{GO}_m(\R)\subset \GL_m(\R)$ denote the group of real invertible matrices preserving the inner product up to a positive scalar factor. Equivalently,
\[
\op{GO}_m(\R)=\{M\in \GL_m(\R): MM^T=\lambda\,\op{Id}_m \text{ for some }\lambda\in\R_{>0}\}.
\]
The space of shapes of rank-$m$ lattices is defined as the double coset space
\[
\mathcal S_m := \GL_m(\Z)\backslash \GL_m(\R)/\op{GO}_m(\R).
\]

Given a $\Z$-basis $b_1,\dots,b_m$ of $\Lambda$, write each basis vector in the form
\[
b_k=\sum_{j=1}^m m_{k,j} e_j,\qquad m_{k,j}\in\R,
\]
and let $[m_{k,j}]\in \GL_m(\R)$ be the associated change-of-basis matrix. The shape of $\Lambda$ is then defined to be the corresponding double coset
\[
\op{sh}(\Lambda):=\GL_m(\Z)\,[m_{k,j}]\,\op{GO}_m(\R)\in\mathcal S_m.
\]
This definition is independent of the choice of $\Z$-basis of $\Lambda$, and it is unchanged if $\Lambda$ is rescaled by a positive real number or acted on by an orthogonal transformation. The space $\mathcal S_m$ carries a natural measure $\mu_m$, obtained by pushing forward a Haar measure on $\GL_m(\R)$.

\par An equivalent and often more convenient description of lattice
shapes is given in terms of Gram matrices.  Let $\mathcal G_m$ denote
the space of positive definite symmetric real $m\times m$ matrices.
The group $\GL_m(\Z)$ acts on $\mathcal G_m$ from the left by $M\cdot G:=MGM^T$ where $M\in\GL_m(\Z)$ and $G\in\mathcal G_m$. This is indeed a left action, since
\[
(M_1M_2)\cdot G
=
M_1\cdot(M_2\cdot G).
\]
The group $\R^\times$ acts by homotheties $G\mapsto r^2G$, where $r\in\R^\times$. There is a natural $\GL_m(\Z)$-equivariant bijection
\[
\GL_m(\R)/\op{GO}_m(\R)
\longrightarrow
\mathcal G_m/\R^\times,
\]
given by $A\mapsto AA^T$.
Consequently, the space of shapes may be identified with
\begin{equation}\label{new interpretation}
\mathcal S_m
=
\GL_m(\Z)\backslash\mathcal G_m/\R^\times.
\end{equation}

If $\Lambda$ has $\Z$-basis $b_1,\dots,b_m$, its Gram matrix is
\[
\op{Gr}(\Lambda)
:=
\bigl(\langle b_k,b_j\rangle\bigr)_{1\leq k,j\leq m}.
\]
If $\mathbf b=(b_1,\dots,b_m)^T$ and $\mathbf b'=M\mathbf b$ for $M\in\GL_m(\Z)$, then
\[
\op{Gr}(\mathbf b')
=
M\,\op{Gr}(\mathbf b)\,M^T.
\]
After an additional scaling by $r\in\R^\times$, the Gram matrix becomes
\[
r^2M\,\op{Gr}(\mathbf b)\,M^T.
\]
Thus its class in
$\GL_m(\Z)\backslash\mathcal G_m/\R^\times$ depends only on the
shape of the lattice.

\subsection{The shape of a number field}
\par Let $L/\Q$ be a number field of degree $n$, and denote by
\[
\sigma_1,\dots,\sigma_n : L \hookrightarrow \C
\]
its embeddings into the complex numbers. The Minkowski map
\[
J : L \longrightarrow \C^n, \qquad
x \mapsto \bigl(\sigma_1(x),\dots,\sigma_n(x)\bigr),
\]
identifies $L$ with a lattice inside a real vector space. More precisely, let
\[
L_\R := \op{Span}_\R J(L) \subset \C^n,
\]
which is an $n$-dimensional real vector space.

We endow $\C^n$ with its standard Hermitian form
\[
\langle z,w\rangle := \sum_{k=1}^n z_k \,\overline{w_k},
\]
and restrict it to $L_\R$. For $x,y\in L$ one computes
\[
\langle J(x),J(y)\rangle
=
\sum_{\tau:L\hookrightarrow\C} \tau(x)\,\overline{\tau(y)}
=
\sum_{\tau:L\hookrightarrow\C} \tau(y)\,\overline{\tau(x)}
=
\langle J(y),J(x)\rangle,
\]
since the embeddings occur in complex conjugate pairs. Hence this restriction is a real, symmetric inner product on $L_\R$.

\par
A natural first attempt to define the shape of $L$ would be to take the shape of the full lattice $J(\cO_L)\subset L_\R$. However, this lattice always contains the distinguished vector
\[
J(1)=(1,1,\dots,1),
\]
which imposes a rigid linear constraint and prevents the resulting shapes from exhibiting genuinely random behavior. To eliminate this obstruction, consider the trace-zero projection
\[
\alpha \mapsto \alpha^\perp := n\alpha - \Tr_{L/\Q}(\alpha),
\]
which maps $\cO_L$ into itself and annihilates $1$. The image
\[
\cO_L^\perp := \{\alpha^\perp : \alpha\in \cO_L\}
\]
is a free $\Z$-module of rank $n-1$, and its Minkowski embedding
\[
J(\cO_L^\perp) \subset L_\R
\]
lies in the orthogonal complement of $J(1)$. This construction produces a rank-$(n-1)$ lattice endowed with an induced inner product.

\begin{definition}
The \emph{shape of the number field $L$} is defined to be the shape of the lattice $J(\cO_L^\perp)$ in the $(n-1)$-dimensional real inner product space $J(1)^\perp \subset L_\R$.
\end{definition}

\par
To describe this shape concretely, choose an integral basis
\[
\{1,\alpha_1,\dots,\alpha_{n-1}\}
\]
of $\cO_L$. Then
\[
\{\alpha_1^\perp,\dots,\alpha_{n-1}^\perp\}
\]
is a $\Z$-basis of $\cO_L^\perp$. The associated Gram matrix is
\[
\op{Gr}\bigl(J(\cO_L^\perp)\bigr)
=
\bigl(\langle J(\alpha_k^\perp),J(\alpha_j^\perp)\rangle\bigr)_{1\le k,j\le n-1},
\]
and its equivalence class in the space of lattice shapes
\[
\mathcal S_{n-1}=\GL_{n-1}(\Z)\backslash \mathcal G/\R^\times
\]
represents the shape of $L$.

\begin{definition}\label{def:relative-shape}
Let $K=\Q(i)$ and let $L/K$ be a quartic extension.  Put
\[
L_{\R,K}^{0}
:=
\left\{
v\in L_{\R}\mid
\langle v,J(1)\rangle=
\langle v,J(i)\rangle=0
\right\}.
\]
This is a real vector space of dimension $6$.  Under the Minkowski
embedding it is naturally identified with the real span of
$\ker(\Tr_{L/K})$.  Indeed, for $x\in L$, the condition
$\Tr_{L/K}(x)=0$ is equivalent to
\[
\Tr_{L/\Q}(x)=\Tr_{L/\Q}(ix)=0.
\]

Define the relative trace-zero lattice by
\[
\cO_{L/K}^{\perp}
:=
\left\{
4\alpha-\Tr_{L/K}(\alpha)\mid 
\alpha\in\cO_L
\right\}.
\]
Since $\Tr_{L/K}(\alpha)\in\Z[i]$ for $\alpha\in\cO_L$, this is a
free $\Z$-module of rank $6$, and
\[
J(\cO_{L/K}^{\perp})\subset L_{\R,K}^{0}.
\]
Up to the harmless scalar factor $4$, this is precisely the
orthogonal projection of $J(\cO_L)$ onto $L_{\R,K}^{0}$.

We define the \emph{relative shape of $L/K$} to be the shape of the
rank-$6$ lattice $J(\cO_{L/K}^{\perp})$.  It therefore determines a
point of $\mathcal S_6$.
\end{definition}

\begin{remark}
\label{rem:pure-quartic-relation}
When $m\in\Z$, there is a direct connection with the pure quartic
fields studied in \cite{purequartic}.  Indeed, if $F=\Q(\sqrt[4]{m})$, then $L=F(i)$. In the pure quartic setting, one studies the lattice
coming from $\cO_F$ after removing the contribution of $1$. In the
present setting, $\cO_L$ contains the larger fixed subring
$\Z[i]$, and we remove the contribution coming from both $1$ and $i$. Writing
\[
m=fg^2h^3
\]
with $f,g,h\in\Z$ squarefree and pairwise coprime, the parameters
\[
\lambda_1=\frac{|f|}{|h|}
\qquad\text{and}\qquad
\lambda_2=|g|
\]
determine the shape in the
pure quartic case \cite{purequartic}.  The present calculation
extends the same basic picture to the Kummer extension
$\Q(i,\sqrt[4]{m})$, with the extra congruence information coming from
the arithmetic of $\Z[i]$.
\end{remark}

\subsection{An integral basis for the fields $\Q(i,\sqrt[4]{m})$}\label{s 1.3}

In this section we recall, for the reader’s convenience, the explicit
description of integral bases for quartic Kummer--cyclotomic extensions
from \cite{quartickummer}.  These results will be used repeatedly in our
subsequent analysis of shapes and their distribution. Let $K=\Q(i)$, $m=fg^{2}h^{3}\in\Z[i]$ where $f,g,h\in\Z[i]$ are squarefree and pairwise coprime, and let $L=K(\alpha)$ with $\alpha^{4}=m$. Since $\Z[i]$ is a principal ideal domain, $\cO_{L}$ is a free
$\Z[i]$-module of rank~$4$.  By the general theory of normalised
integral bases, there exists a basis of the form
\begin{equation}\label{eq:nib}
\Bigl\{
1,\;
\frac{f_{1}(\alpha)}{d_{1}},\;
\frac{f_{2}(\alpha)}{d_{2}},\;
\frac{f_{3}(\alpha)}{d_{3}}
\Bigr\},
\end{equation}
where $d_{i}\in\Z[i]$ and $f_{j}(X)\in\Z[i][X]$ are monic polynomials of
degree~$j$. In \cite{quartickummer}, all possible triples $(d_{1},d_{2},d_{3})$ are
determined explicitly in terms of the residue classes of $f,g,h$ modulo
powers of $(1+i)$.  The outcome is a finite list of cases.
\par A key role is played by $M:= K(\sqrt{fh}) \subset L$.
Letting $\gamma:=\sqrt{fh}$, \cite[Proposition 1]{quartickummer} asserts that 
$M/K$ admits an integral basis of the form $\Bigl\{1,\;\frac{a+\gamma}{d_{1}}\Bigr\}$,
where $d_{1}^2|4$ and $a^{2}\equiv fh \pmod{d_{1}^{2}}$. The precise value of $d_{1}$ is determined in \cite[Table 1]{quartickummer}. Theorem 2 in \emph{loc. cit.} shows that, depending on the
congruence classes of $f,g,h$ modulo $(1+i)$, the normalised
integral basis can always be chosen with $(f_1, f_2, f_3, d_1, d_2, d_3)$ belonging to an explicit
list, given below:
\begin{table}[ht]
\centering
\scriptsize
\renewcommand{\arraystretch}{1.15}
\caption{Integral bases for $L=\Q(i,\sqrt[4]{m})$}
\label{tab:integral-bases}
\begin{adjustbox}{max width=\textwidth}
\begin{tabular}{|c|p{6cm}|p{9cm}|}
\hline
\textbf{Case} & \textbf{Conditions on $m,f,g,h$} & \textbf{Integral basis of $\mathcal O_L$ over $\Z[i]$} \\
\hline
1
&
$m \equiv 1 \pmod{8}$
&
$\left\{
1,\;
\dfrac{1+\alpha}{1+i},\;
\dfrac{|gh|^{2}\bigl(i+(1+i)\alpha\bigr)+\alpha^{2}}{2(1+i)gh},\;
\dfrac{|gh^{2}|^{2}\bigl(1+\alpha+\alpha^{2}\bigr)+\alpha^{3}}{4gh^{2}}
\right\}$ \\
\hline
2
&
$m \equiv 1+4i \pmod{8}$
&
$\left\{
1,\;
\dfrac{1+\alpha}{1+i},\;
\dfrac{|gh|^{2}\bigl(-i+(1+i)\alpha\bigr)+\alpha^{2}}{2(1+i)gh},\;
\dfrac{|gh^{2}|^{2}\bigl(2-i+\alpha+i\alpha^{2}\bigr)+\alpha^{3}}{4gh^{2}}
\right\}$ \\
\hline
3
&
$\begin{aligned}
m &\equiv 2i \pmod{4},\\
fh &\equiv 1 \pmod{4}
\end{aligned}$
&
$\left\{
1,\;
\alpha,\;
\dfrac{gh+\alpha^{2}}{2gh},\;
\dfrac{\alpha+\alpha^{3}/gh^{2}}{2}
\right\}$ \\
\hline
4
&
$\begin{aligned}
m &\equiv 2i \pmod{4},\\
fh &\equiv -1 \pmod{4}
\end{aligned}$
&
$\left\{
1,\;
\alpha,\;
\dfrac{i gh+\alpha^{2}}{2gh},\;
\dfrac{i\alpha+\alpha^{3}/gh^{2}}{2}
\right\}$ \\
\hline
5
&
$\begin{aligned}
m &\equiv 2i \pmod{4},\\
fh &\equiv \pm1 \pmod{2(1+i)},\\
f\bar h &\equiv 1 \pmod{2(1+i)}
\end{aligned}$
&
$\left\{
1,\;
\alpha,\;
\dfrac{gh+\alpha^{2}}{(1+i)gh},\;
\dfrac{\alpha+\alpha^{3}/gh^{2}}{2}
\right\}$ \\
\hline
6
&
$\begin{aligned}
m &\equiv 2i \pmod{4},\\
fh &\equiv \pm1 \pmod{2(1+i)},\\
f\bar h &\equiv -1 \pmod{2(1+i)}
\end{aligned}$
&
$\left\{
1,\;
\alpha,\;
\dfrac{gh+\alpha^{2}}{(1+i)gh},\;
\dfrac{i\alpha+\alpha^{3}/gh^{2}}{2}
\right\}$ \\
\hline
7
&
$m \equiv 3+2i \pmod{4}$
&
$\left\{
1,\;
\alpha,\;
\dfrac{|gh|^{2}+\alpha^{2}}{(1+i)gh},\;
\dfrac{|gh^{2}|^{2}\bigl(1+\alpha+\alpha^{2}\bigr)+\alpha^{3}}{2gh^{2}}
\right\}$ \\
\hline
8
&
$m \equiv 1+2i \pmod{4}$
&
$\left\{
1,\;
\alpha,\;
\dfrac{|gh|^{2}\bigl(1+(1+i)\alpha\bigr)+\alpha^{2}}{2gh},\;
\dfrac{|gh^{2}|^{2}\bigl(\alpha+(1+i)\alpha^{2}\bigr)+\alpha^{3}}{2gh^{2}}
\right\}$ \\
\hline
9
&
$m \equiv 3 \pmod{4}$
&
$\left\{
1,\;
\alpha,\;
\dfrac{i|gh|^{2}+\alpha^{2}}{2gh},\;
\dfrac{|gh^{2}|^{2}\bigl(i+i\alpha+\alpha^{2}\bigr)+\alpha^{3}}{2gh^{2}}
\right\}$ \\
\hline
10
&
$\begin{aligned}
m &\equiv 5 \pmod{8},\\
\text{or } m &\equiv 5+4i \pmod{8}
\end{aligned}$
&
$\left\{
1,\;
\dfrac{1+\alpha}{1+i},\;
\dfrac{|gh|^{2}+\alpha^{2}}{2gh},\;
\dfrac{|gh^{2}|^{2}\bigl(1+\alpha+\alpha^{2}\bigr)+\alpha^{3}}{2(1+i)gh^{2}}
\right\}$ \\
\hline
11
&
$f$ even or $h$ even or $m \equiv i \pmod{2}$
&
$\left\{
1,\;
\alpha,\;
\dfrac{\alpha^{2}}{gh},\;
\dfrac{\alpha^{3}}{gh^{2}}
\right\}$ \\
\hline
12
&
$\begin{aligned}
fh &\equiv i \pmod{2},\\
g &\text{ even}
\end{aligned}$
&
$\left\{
1,\;
\alpha,\;
\dfrac{\alpha^{2}}{gh},\;
\dfrac{i\alpha+\alpha^{3}/gh^{2}}{1+i}
\right\}$ \\
\hline
\end{tabular}
\end{adjustbox}
\end{table}

\section{Computing Gram matrices}
\par
With the notation of Section~\ref{s 1.3}, let $K=\Q(i)$ and $L=K(\alpha)$, where $\alpha^4=m\in\Z[i]$ is fourth-power free. Write $m=fg^2h^3$,
with $f,g,h\in\Z[i]$ pairwise coprime and squarefree. The embeddings of $L$ into $\C$ are parametrized by $(k,\varepsilon)\in\{0,1,2,3\}\times\{\pm\}$
defined by setting $\tau_{k,\varepsilon}(i)=\varepsilon i$ and $\tau_{k,\varepsilon}(\alpha)=i^k\alpha$.
We shall write 
\[J(\alpha)=\left(\tau_{0, +}(\alpha), \dots, \tau_{3, +}(\alpha), \tau_{0, -}(\alpha), \dots, \tau_{3, -}(\alpha)\right),\] and thus in particular
\[J(i)=\left(i,i,i,i, -i, -i, -i, -i\right).\]
\subsection{Gram matrices of Type I}
\par We say that $L$ is of \emph{Type I} if we are in one of the cases $3,4,5,6,11$ or $12$. The integral bases computed in each of these cases can be related by a fixed change of basis matrix. For simplicity, we begin with Case 11. In this case, an integral-basis of $\cO_L$ is given by
\[
\mathcal B_{\Z}
=
\left\{
e_1,e_2,e_3,e_4,e_5,e_6,e_7,e_8
\right\}
=
\left\{
1,\ \alpha,\ \frac{\alpha^2}{gh},\ \frac{\alpha^3}{gh^2},
\ i,\ i\alpha,\ i\frac{\alpha^2}{gh},\ i\frac{\alpha^3}{gh^2}
\right\}.
\]
For $r,s\in\{1,\dots,8\}$ we have
\[
\langle e_r,e_s\rangle
=
\sum_{k=0}^3
\Bigl(
\tau_{k,+}(e_r)\,\overline{\tau_{k,+}(e_s)}
+
\tau_{k,-}(e_r)\,\overline{\tau_{k,-}(e_s)}
\Bigr).
\]
\begin{proposition}
Suppose that $f$ is even or $h$ is even or $m \equiv i \pmod{2}$ and the conditions imposed by Case 11 are satisfied. Then the Gram matrix $G^{(8)}_{11}
=
\bigl(\langle e_r,e_s\rangle\bigr)_{1\le r,s\le 8}$ is diagonal and given by
\[
\begin{pmatrix}
8 & 0 & 0 & 0 & 0 & 0 & 0 & 0 \\
0 & 8|m|^{1/2} & 0 & 0 & 0 & 0 & 0 & 0 \\
0 & 0 & \dfrac{8|m|}{|gh|^{2}} & 0 & 0 & 0 & 0 & 0 \\
0 & 0 & 0 & \dfrac{8|m|^{3/2}}{|gh^2|^{2}} & 0 & 0 & 0 & 0 \\
0 & 0 & 0 & 0 & 8 & 0 & 0 & 0 \\
0 & 0 & 0 & 0 & 0 & 8|m|^{1/2} & 0 & 0 \\
0 & 0 & 0 & 0 & 0 & 0 & \dfrac{8|m|}{|gh|^{2}} & 0 \\
0 & 0 & 0 & 0 & 0 & 0 & 0 & \dfrac{8|m|^{3/2}}{|gh^2|^{2}}
\end{pmatrix}
\]

\end{proposition}

\begin{proof}
To determine the shape of $\cO_L$, we compute the Gram matrix of $\mathcal B_{\Z}$. Note that for $1\le r\le4$, each basis element has the form
\[
e_r=c_r\,\alpha^{a_r},
\]
where
\[
(a_r,c_r)=
\begin{cases}
(0,1), & r=1,\\
(1,1), & r=2,\\
(2,(gh)^{-1}), & r=3,\\
(3,(gh^2)^{-1}), & r=4,
\end{cases}
\qquad c_r\in K=\Q(i).
\]
and thus, we find that $\tau_{k,\varepsilon}(e_r)
=
c_r\,i^{k a_r}\alpha^{a_r}$. For $r=s$, a direct computation gives
\[
\langle 1,1\rangle=8,
\qquad
\langle \alpha,\alpha\rangle=8|\alpha|^2=8|m|^{1/2},
\]
\[
\left\langle \frac{\alpha^2}{gh},\frac{\alpha^2}{gh}\right\rangle
=
8\frac{|\alpha|^4}{|gh|^2}
=
8\frac{|m|}{|gh|^2},
\]
and
\[
\left\langle \frac{\alpha^3}{gh^2},\frac{\alpha^3}{gh^2}\right\rangle
=
8\frac{|\alpha|^6}{|gh^2|^2}
=
8\frac{|m|^{3/2}}{|gh^2|^2}.
\]
\noindent If $1\le r\neq s\le4$, then
\[
\langle e_r,e_s\rangle
=
2c_r\overline{c_s}\alpha^{a_r}\overline{\alpha^{a_s}}
\sum_{k=0}^3 i^{k(a_r-a_s)}=0,
\]
since $a_r\not\equiv a_s\pmod{4}$. Consider the diagonal block $5\le r,s\le8$, for which the entries coincide with the previous ones, since
\[
\tau_{k,\varepsilon}(ie_j)=\varepsilon i\,\tau_{k,\varepsilon}(e_j).
\]
Therefore, we find that
\[
\langle e_{4+r},e_{4+s}\rangle
=
\langle e_r,e_s\rangle
\]
for $1\le r,s\le4$.
\par It is easy to see that for $1\leq r,s\leq 4$,
\[
\langle e_r,e_{4+s}\rangle
=
\sum_{k=0}^3
\Bigl(
\tau_{k,+}(e_r)\,\overline{i\tau_{k,+}(e_s)}
+
\tau_{k,-}(e_r)\,\overline{-i\tau_{k,-}(e_s)}
\Bigr)
=0,
\]
so all mixed entries vanish. Collecting all $64$ entries yields the stated diagonal Gram matrix.
\end{proof}
\par
In every case under consideration, the ring of integers $\cO_L$ contains
$\cO_K=\Z[i]$. Consequently, the Minkowski lattice $J(\cO_L)\subset L_\R$
always contains the two distinguished vectors $J(1)$ and $J(i)$.
These vectors impose rigid linear relations which obscure the intrinsic
geometry of the remaining directions. For this reason, we define
$J(\cO_L)'$ to be the orthogonal projection of $J(\cO_L)$ onto the real
subspace
\[
L_\R' := \{\, v\in L_\R \mid \langle v, J(1)\rangle = \langle v, J(i)\rangle =0 \,\},
\]
which has codimension~$2$ in $L_\R$. In Case~11, the Gram matrix computed above shows that the shape of the
projected lattice $J(\cO_L)'$ is 
\[
\begin{pmatrix}
1 & 0 & 0 & 0 & 0 & 0\\
0 & \dfrac{|m|^{1/2}}{|gh|^{2}} & 0 & 0 & 0 & 0\\
0 & 0 & \dfrac{|m|}{|gh^2|^{2}} & 0 & 0 & 0\\
0 & 0 & 0 & 1 & 0 & 0\\
0 & 0 & 0 & 0 & \dfrac{|m|^{1/2}}{|gh|^{2}} & 0\\
0 & 0 & 0 & 0 & 0 & \dfrac{|m|}{|gh^2|^{2}}
\end{pmatrix}
\]
and is completely determined by the diagonal entries
\[
\frac{c}{b}
=
\frac{|m|^{1/2}}{|gh|^{2}}
=
\sqrt{\frac{|fg^2h^3|}{|gh|^{4}}}
=
\frac{\sqrt{|f|/|h|}}{|g|},
\qquad
\frac{d}{b}
=
\frac{|m|}{|gh^2|^{2}}
=
\frac{|fg^2h^3|}{|gh^2|^{2}}
=
\frac{|f|}{|h|}.
\]
\begin{definition}\label{shape parameter defn}Motivated by the above observation, we introduce the following
\emph{shape parameters}:
\[
\lambda_1 := \frac{|f|}{|h|},
\qquad
\lambda_2 := |g|,
\]
analogous to the parameters considered by Holmes \cite[Section 4.2.1]{Holmes}.
The pair $(\lambda_1,\lambda_2)$ uniquely determines the shape of the
projected lattice $J(\cO_L)'$ in Case 11.
\end{definition}

\par Likewise, we may obtain expressions for the projected shape matrices in the other cases whose formulae involve $\lambda_1$ and $\lambda_2$. Let $G_{11}^{(4)}$ denote the $4\times 4$ Gram matrix associated to the $\Z[i]$--basis
\[
\{\,1,\ \alpha,\ \alpha^{2}/(gh),\ \alpha^{3}/(gh^{2})\,\}.
\]
By the explicit calculations above, this matrix is diagonal and given by
\[
G_{11}^{(4)}
=
\begin{pmatrix}
8 & 0 & 0 & 0 \\
0 & 8|m|^{1/2} & 0 & 0 \\
0 & 0 & \dfrac{8|m|}{|gh|^{2}} & 0 \\
0 & 0 & 0 & \dfrac{8|m|^{3/2}}{|gh^2|^{2}}
\end{pmatrix},
\]
and moreover, the full $8\times 8$ matrix decomposes as
\[
G_{11}^{(8)}
=
\begin{pmatrix}
G_{11}^{(4)} & 0 \\
0 & G_{11}^{(4)}
\end{pmatrix}.
\]
The shape is represented by 
\[\begin{pmatrix}
\frac{1}{8|m|^{1/2}}\cdot  G_{11}^{(4)} & 0 \\
0 & \frac{1}{8|m|^{1/2}}\cdot G_{11}^{(4)}
\end{pmatrix}.\]
To explicitly demonstrate the dependence of the lattice shapes on the parameters $\lambda_1$ and $\lambda_2$, we establish the following lemma. 
\begin{lemma}\label{lem:iden}
Let $\lambda_1 = |f|/|h|$ and $\lambda_2 = |g|$. The following identities express the fundamental quantities of the shape analysis in terms of $|m|$ and the parameters $(\lambda_1, \lambda_2)$:
\begin{align*}
    |h| &= \frac{|m|^{1/4}}{\lambda_1^{1/4} \lambda_2^{1/2}} &
    |f||h| &= \frac{\sqrt{\lambda_1}}{\lambda_2} |m|^{1/2} \\
    |gh| &= \frac{\sqrt{\lambda_2}}{\lambda_1^{1/4}} |m|^{1/4} &
    |gh^2| &= \frac{|m|^{1/2}}{\sqrt{\lambda_1}} \\
    |g|^2|h|^2 &= \frac{\lambda_2}{\sqrt{\lambda_1}} |m|^{1/2} &
    |g|^2|h|^4 &= \frac{|m|}{\lambda_1}
\end{align*}
\end{lemma}

\begin{proof}
Using the relation $|m| = |f||g|^2|h|^3$ along with the definitions $\lambda_1 = |f|/|h|$ and $\lambda_2 = |g|$, the equalities follow immediately by direct substitution.
\end{proof}
\noindent Thus we find that 
\[
\frac{1}{8|m|^{1/2}}\cdot  G_{11}^{(4)}=\begin{pmatrix}
\frac{1}{|m|^{1/2}} & 0 & 0 & 0 \\
0 & 1 & 0 & 0 \\
0 & 0 & \frac{\sqrt{\lambda_1}}{\lambda_2} & 0 \\
0 & 0 & 0 & \lambda_1
\end{pmatrix},
\]
and thus as $|m|\rightarrow \infty$, the $(1,1)$ entry of the above matrix approaches $0$. Thus in the limit $|m|\rightarrow \infty$, the shape matrix approaches a diagonal matrix whose entries are determined by $\lambda_1$ and $\lambda_2$.

\begin{proposition} \label{prop:gram_cases_other}
Let $G_{11}^{(8)}= \text{diag}(8, 8A, 8B, 8A\lambda, 8, 8A, 8B, 8A\lambda)$ denote the diagonal Gram matrix associated with the basis of Case 11, where $A = |m|^{1/2}, \quad B = |fh|, \quad \lambda = \frac{|f|}{|h|}$. For each case $k \in \{3, 4, 5, 6, 12\}$, let $C_k$ denote the transition matrix representing the change of basis from the Case 11 basis to the Case $k$ basis. The corresponding Gram matrix $G_k^{(8)}$ for Case $k$ is given by $G_k^{(8)} = C_k^T G_{11}^{(8)} C_k,$
where $C_k$ and $G_{k}^{(8)}$ for $k\in \{3,4,5,6,12\}$ are given by

\begin{minipage}[t]{0.4\textwidth}
    \vspace{5pt}
    { $C_{12}$ =  
    $
    \begin{pmatrix}
    1 & 0 & 0 & 0 & 0 & 0 & 0 & 0 \\
    0 & 1 & 0 & \frac{1}{2} & 0 & 0 & 0 & -\frac{1}{2} \\
    0 & 0 & 1 & 0 & 0 & 0 & 0 & 0 \\
    0 & 0 & 0 & \frac{1}{2} & 0 & 0 & 0 & \frac{1}{2} \\
    0 & 0 & 0 & 0 & 1 & 0 & 0 & 0 \\
    0 & 0 & 0 & \frac{1}{2} & 0 & 1 & 0 & \frac{1}{2} \\
    0 & 0 & 0 & 0 & 0 & 0 & 1 & 0 \\
    0 & 0 & 0 & -\frac{1}{2} & 0 & 0 & 0 & \frac{1}{2} 
    \end{pmatrix}
    $
    }
\end{minipage}
\hfill
\begin{minipage}[t]{0.6\textwidth}
    \centering
    \vspace{5pt}
    {$G_{12}^{(8)}$ =  
    $
    \begin{pmatrix}
    8 & 0 & 0 & 0 & 0 & 0 & 0 & 0 \\
    0 & 8A & 0 & 4A & 0 & 0 & 0 & -4A \\
    0 & 0 & 8B & 0 & 0 & 0 & 0 & 0 \\
    0 & 4A & 0 & 4A(1+\lambda) & 0 & 4A & 0 & 0 \\
    0 & 0 & 0 & 0 & 8 & 0 & 0 & 0 \\
    0 & 0 & 0 & 4A & 0 & 8A & 0 & 4A \\
    0 & 0 & 0 & 0 & 0 & 0 & 8B & 0 \\
    0 & -4A & 0 & 0 & 0 & 4A & 0 & 4A(1+\lambda)
    \end{pmatrix}
    $
    }
\end{minipage}

\vspace{0.5cm}
\begin{minipage}[t]{0.37\textwidth}
    \vspace{5pt}
    { 
    $ C_{3}=
    \begin{pmatrix}
    1 & 0 & \frac{1}{2} & 0 & 0 & 0 & 0 & 0 \\
    0 & 1 & 0 & \frac{1}{2} & 0 & 0 & 0 & 0 \\
    0 & 0 & \frac{1}{2} & 0 & 0 & 0 & 0 & 0 \\
    0 & 0 & 0 & \frac{1}{2} & 0 & 0 & 0 & 0 \\
    0 & 0 & 0 & 0 & 1 & 0 & \frac{1}{2} & 0 \\
    0 & 0 & 0 & 0 & 0 & 1 & 0 & \frac{1}{2} \\
    0 & 0 & 0 & 0 & 0 & 0 & \frac{1}{2} & 0 \\
    0 & 0 & 0 & 0 & 0 & 0 & 0 & \frac{1}{2} 
    \end{pmatrix}
    $
    }
\end{minipage}
\hfill
\begin{minipage}[t]{0.63\textwidth}
    \vspace{5pt}
    {
    $G_{3}^{(8)}=
    \begin{pmatrix}
    8 & 0 & 4 & 0 & 0 & 0 & 0 & 0 \\
    0 & 8A & 0 & 4A & 0 & 0 & 0 & 0 \\
    4 & 0 & 2+2B & 0 & 0 & 0 & 0 & 0 \\
    0 & 4A & 0 & 2A(1+\lambda) & 0 & 0 & 0 & 0 \\
    0 & 0 & 0 & 0 & 8 & 0 & 4 & 0 \\
    0 & 0 & 0 & 0 & 0 & 8A & 0 & 4A \\
    0 & 0 & 0 & 0 & 4 & 0 & 2+2B & 0 \\
    0 & 0 & 0 & 0 & 0 & 4A & 0 & 2A(1+\lambda)
    \end{pmatrix}
    $
    }
\end{minipage}

\vspace{0.5cm}
\begin{minipage}[t]{0.38\textwidth}
    \vspace{5pt}
    { 
    $C_4 = 
    \begin{pmatrix}
    1 & 0 & 0 & 0 & 0 & 0 & -\frac{1}{2} & 0 \\
    0 & 1 & 0 & 0 & 0 & 0 & 0 & -\frac{1}{2} \\
    0 & 0 & \frac{1}{2} & 0 & 0 & 0 & 0 & 0 \\
    0 & 0 & 0 & \frac{1}{2} & 0 & 0 & 0 & 0 \\
    0 & 0 & \frac{1}{2} & 0 & 1 & 0 & 0 & 0 \\
    0 & 0 & 0 & \frac{1}{2} & 0 & 1 & 0 & 0 \\
    0 & 0 & 0 & 0 & 0 & 0 & \frac{1}{2} & 0 \\
    0 & 0 & 0 & 0 & 0 & 0 & 0 & \frac{1}{2} 
    \end{pmatrix}
    $
    }
\end{minipage}
\hfill
\begin{minipage}[t]{0.62\textwidth}
    \vspace{5pt}
    { 
    $G_{4}^{(8)} = 
    \begin{pmatrix}
    8 & 0 & 0 & 0 & 0 & 0 & -4 & 0 \\
    0 & 8A & 0 & 0 & 0 & 0 & 0 & -4A \\
    0 & 0 & 2+2B & 0 & 4 & 0 & 0 & 0 \\
    0 & 0 & 0 & 2A(1+\lambda) & 0 & 4A & 0 & 0 \\
    0 & 0 & 4 & 0 & 8 & 0 & 0 & 0 \\
    0 & 0 & 0 & 4A & 0 & 8A & 0 & 0 \\
    -4 & 0 & 0 & 0 & 0 & 0 & 2+2B & 0 \\
    0 & -4A & 0 & 0 & 0 & 0 & 0 & 2A(1+\lambda)
    \end{pmatrix}
    $
    }
\end{minipage}

\vspace{0.5cm}
\begin{minipage}[t]{0.38\textwidth}
    \vspace{5pt}
    { 
    $C_5 = 
    \begin{pmatrix}
    1 & 0 & \frac{1}{2} & 0 & 0 & 0 & \frac{1}{2} & 0 \\
    0 & 1 & 0 & \frac{1}{2} & 0 & 0 & 0 & 0 \\
    0 & 0 & \frac{1}{2} & 0 & 0 & 0 & \frac{1}{2} & 0 \\
    0 & 0 & 0 & \frac{1}{2} & 0 & 0 & 0 & 0 \\
    0 & 0 & -\frac{1}{2} & 0 & 1 & 0 & \frac{1}{2} & 0 \\
    0 & 0 & 0 & 0 & 0 & 1 & 0 & \frac{1}{2} \\
    0 & 0 & -\frac{1}{2} & 0 & 0 & 0 & \frac{1}{2} & 0 \\
    0 & 0 & 0 & 0 & 0 & 0 & 0 & \frac{1}{2} 
    \end{pmatrix}
    $
    }
\end{minipage}
\hfill
\begin{minipage}[t]{0.62\textwidth}
    \vspace{5pt}
    { 
    $G_{5}^{(8)} = 
    \begin{pmatrix}
    8 & 0 & 4 & 0 & 0 & 0 & 4 & 0 \\
    0 & 8A & 0 & 4A & 0 & 0 & 0 & 0 \\
    4 & 0 & 4+4B & 0 & -4 & 0 & 0 & 0 \\
    0 & 4A & 0 & 2A(1+\lambda) & 0 & 0 & 0 & 0 \\
    0 & 0 & -4 & 0 & 8 & 0 & 4 & 0 \\
    0 & 0 & 0 & 0 & 0 & 8A & 0 & 4A \\
    4 & 0 & 0 & 0 & 4 & 0 & 4+4B & 0 \\
    0 & 0 & 0 & 0 & 0 & 4A & 0 & 2A(1+\lambda)
    \end{pmatrix}
    $
    }
\end{minipage}
\vspace{0.5cm}
\begin{minipage}[t]{0.38\textwidth}
    \vspace{5pt}
    { 
    $C_6 = 
    \begin{pmatrix}
    1 & 0 & \frac{1}{2} & 0 & 0 & 0 & \frac{1}{2} & 0 \\
    0 & 1 & 0 & 0 & 0 & 0 & 0 & -\frac{1}{2} \\
    0 & 0 & \frac{1}{2} & 0 & 0 & 0 & \frac{1}{2} & 0 \\
    0 & 0 & 0 & \frac{1}{2} & 0 & 0 & 0 & 0 \\
    0 & 0 & -\frac{1}{2} & 0 & 1 & 0 & \frac{1}{2} & 0 \\
    0 & 0 & 0 & \frac{1}{2} & 0 & 1 & 0 & 0 \\
    0 & 0 & -\frac{1}{2} & 0 & 0 & 0 & \frac{1}{2} & 0 \\
    0 & 0 & 0 & 0 & 0 & 0 & 0 & \frac{1}{2} 
    \end{pmatrix}
    $
    }
\end{minipage}
\hfill
\begin{minipage}[t]{0.62\textwidth}
    \vspace{5pt}
    { 
    $G_{6}^{(8)} = 
    \begin{pmatrix}
    8 & 0 & 4 & 0 & 0 & 0 & 4 & 0 \\
    0 & 8A & 0 & 0 & 0 & 0 & 0 & -4A \\
    4 & 0 & 4+4B & 0 & -4 & 0 & 0 & 0 \\
    0 & 0 & 0 & 2A(1+\lambda) & 0 & 4A & 0 & 0 \\
    0 & 0 & -4 & 0 & 8 & 0 & 4 & 0 \\
    0 & 0 & 0 & 4A & 0 & 8A & 0 & 0 \\
    4 & 0 & 0 & 0 & 4 & 0 & 4+4B & 0 \\
    0 & -4A & 0 & 0 & 0 & 0 & 0 & 2A(1+\lambda)
    \end{pmatrix}
    $
    }
\end{minipage}
\end{proposition}

\begin{proof}
    The proof follows via a direct computation of transition matrices and associated Gram matrices.
\end{proof}
\subsection{Gram matrices of Type II}  We say that $L$ is of \emph{Type II} if we are in one of the cases $1,2,7,8,9$ or $10$.
\noindent We compute below the Gram matrices \( G_{k}^{(8)} \) together with the corresponding transition matrices \( C_{k} \) for \( k \in \{1,2,7,8,9,10\} \). 

\begin{proposition}
Let $L = \mathbb{Q}(i, \sqrt[4]{m})$ be an octic Kummer extension satisfying Case 7 conditions  $(m \equiv 3+2i \pmod 4)$. Let $\mathcal{B}_{\mathbb{Z}} = \{u_1, \dots, u_4, \, i u_1, \dots, i u_4\}$ be the integral basis defined in Table 1.

The transition matrix $C_{7}$ expresses the Case 7 basis $\mathcal{B}_{\mathbb{Z}}$ in terms of the Case 11 basis, is given by:

\[
C_{7} = \begin{pmatrix}
1 & 0 & \text{Re}(c_1) & \text{Re}(c_2) & 0 & 0 & -\text{Im}(c_1) & -\text{Im}(c_2) \\ 
0 & 1 & 0 & \text{Re}(c_2) & 0 & 0 & 0 & -\text{Im}(c_2) \\ 
0 & 0 & 1/2 & \text{Re}(c_2 gh) & 0 & 0 & 1/2 & -\text{Im}(c_2 gh) \\ 
0 & 0 & 0 & 1/2 & 0 & 0 & 0 & 0 \\ 
0 & 0 & \text{Im}(c_1) & \text{Im}(c_2) & 1 & 0 & \text{Re}(c_1) & \text{Re}(c_2) \\ 
0 & 0 & 0 & \text{Im}(c_2) & 0 & 1 & 0 & \text{Re}(c_2) \\ 
0 & 0 & -1/2 & \text{Im}(c_2 gh) & 0 & 0 & 1/2 & \text{Re}(c_2 gh) \\ 
0 & 0 & 0 & 0 & 0 & 0 & 0 & 1/2 
\end{pmatrix}.
\]
where $c_1, c_2$ are defined as 
$$c_1 = \frac{\bar{g}\bar{h}}{1+i} \mbox{ and } c_2 = \frac{\bar{g}\bar{h}^2}{2}.$$
The Gram matrix $G_{7}^{(8)} = (\langle b_r, b_s \rangle)_{1\le r,s \le 8}$ associated with the ordered basis $\mathcal{B}_{\mathbb{Z}}$ is the block matrix
\[
G_{7}^{(8)} = \begin{pmatrix} M & N \\ -N & M \end{pmatrix},
\]
where $M$ is symmetric and $N$ is antisymmetric. The block entries are given by:
\[
M = \begin{pmatrix}
8 & 0 & 4\text{Re}(gh(1+i)) & 4\text{Re}(gh^2) \\
0 & 8|m|^{1/2} & 0 & 4|m|^{1/2}\text{Re}(gh^2) \\
4\text{Re}(gh(1+i)) & 0 & 4(|g|^2|h|^2 + |f||h|) & 2\Omega_7\text{Re}(h(1-i)) \\
4\text{Re}(gh^2) & 4|m|^{1/2}\text{Re}(gh^2) & 2\Omega_7\text{Re}(h(1-i)) & \Sigma_7
\end{pmatrix}
\]
and
\[
N = \begin{pmatrix}
0 & 0 & 4\text{Im}(gh(1+i)) & 4\text{Im}(gh^2) \\
0 & 0 & 0 & 4|m|^{1/2}\text{Im}(gh^2) \\
-4\text{Im}(gh(1+i)) & 0 & 0 & 2\Omega_7\text{Im}(h(1-i)) \\
-4\text{Im}(gh^2) & -4|m|^{1/2}\text{Im}(gh^2) & -2\Omega_7\text{Im}(h(1-i)) & 0
\end{pmatrix},
\]
where $\Omega_7$ and $\Sigma_7$ are given by:
\begin{align*}
\Omega_7 &= |g|^2|h|^2 + |m|, \\
\Sigma_7 &= 2|g|^2|h|^4 (1 + |m|^{1/2} + |m|) + \frac{2|m|^{3/2}}{|g|^2|h|^4}.
\end{align*}
\end{proposition}

\begin{proposition}
Let $L = \mathbb{Q}(i, \sqrt[4]{m})$ be an octic Kummer extension satisfying Case 10 conditions $(m \equiv 5 \pmod 8$ or $m \equiv 5+4i \pmod 8)$. Let $\mathcal{B}_{\mathbb{Z}} = \{u_1, \dots, u_4, \, i u_1, \dots, i u_4\}$ be the integral basis defined in Table 1.

The transition matrix $C_{10}$ expresses the Case 10 basis $\mathcal{B}_{\mathbb{Z}}$ in terms of the Case 7 basis, is given by:

\[
C_{10} = \begin{pmatrix}
1 & 1/2 & 0 & 0 & 0 & 1/2 & 0 & 0 \\
0 & 1/2 & 0 & 0 & 0 & 1/2 & 0 & 0 \\
0 & 0 & 1/2 & 0 & 0 & 0 & -1/2 & 0 \\
0 & 0 & 0 & 1/2 & 0 & 0 & 0 & 1/2 \\
0 & -1/2 & 0 & 0 & 1 & 1/2 & 0 & 0 \\
0 & -1/2 & 0 & 0 & 0 & 1/2 & 0 & 0 \\
0 & 0 & 1/2 & 0 & 0 & 0 & 1/2 & 0 \\
0 & 0 & 0 & -1/2 & 0 & 0 & 0 & 1/2
\end{pmatrix}.
\]

The Gram matrix $G_{10}^{(8)} = (\langle b_r, b_s \rangle)_{1\le r,s \le 8}$ associated with the ordered basis $\mathcal{B}_{\mathbb{Z}}$ is the block matrix
\[
G_{10}^{(8)} = \begin{pmatrix} M & N \\ -N & M \end{pmatrix},
\]
where $M$ is symmetric and $N$ is antisymmetric. The block entries are given by:
\[
M = \begin{pmatrix}
8 & 4 & 4\text{Re}(gh) & 2\text{Re}(gh^2(1+i)) \\
4 & 4(1+|m|^{1/2}) & 2\text{Re}(gh(1-i)) & 2(1+|m|^{1/2})\text{Re}(gh^2) \\
4\text{Re}(gh) & 2\text{Re}(gh(1-i)) & 2|g|^2|h|^2 + 2|f||h| & \Omega_{10} \text{Re}(h(1+i)) \\
2\text{Re}(gh^2(1+i)) & 2(1+|m|^{1/2})\text{Re}(gh^2) & \Omega_{10} \text{Re}(h(1+i)) & \Sigma_{10}
\end{pmatrix}
\]
and
\[
N = \begin{pmatrix}
0 & 4 & 4\text{Im}(gh) & 2\text{Im}(gh^2(1+i)) \\
-4 & 0 & 2\text{Im}(gh(1-i)) & 2(1+|m|^{1/2})\text{Im}(gh^2) \\
-4\text{Im}(gh) & -2\text{Im}(gh(1-i)) & 0 & \Omega_{10} \text{Im}(h(1+i)) \\
-2\text{Im}(gh^2(1+i)) & -2(1+|m|^{1/2})\text{Im}(gh^2) & -\Omega_{10} \text{Im}(h(1+i)) & 0
\end{pmatrix},
\]
where $\Omega_{10}$ and $\Sigma_{10}$ are given by:
\begin{align*}
\Omega_{10} &= |g|^2|h|^2 + |m|, \\
\Sigma_{10} &= |g|^2|h|^4 (1 + |m|^{1/2} + |m|) + \frac{|m|^{3/2}}{|g|^2|h|^4}.
\end{align*}
\end{proposition}

\begin{proposition}
Let $L = \mathbb{Q}(i, \sqrt[4]{m})$ be an octic Kummer extension satisfying Case 8 conditions $(m \equiv 1+2i \pmod 4)$. Let $\mathcal{B}_{\mathbb{Z}} = \{u_1, \dots, u_4, \, i u_1, \dots, i u_4\}$ be the integral basis defined in Table 1.

The transition matrix $C_{8}$ expresses the Case 8 basis $\mathcal{B}_{\mathbb{Z}}$ in terms of the Case 11 basis, is given by:

\[
C_{8} = \begin{pmatrix}
1 & 0 & \text{Re}(c_1) & 0 & 0 & 0 & -\text{Im}(c_1) & 0 \\
0 & 1 & \text{Re}(c_2) & \text{Re}(c_3) & 0 & 0 & -\text{Im}(c_2) & -\text{Im}(c_3) \\
0 & 0 & 1/2 & \text{Re}(c_4) & 0 & 0 & 0 & -\text{Im}(c_4) \\
0 & 0 & 0 & 1/2 & 0 & 0 & 0 & 0 \\
0 & 0 & \text{Im}(c_1) & 0 & 1 & 0 & \text{Re}(c_1) & 0 \\
0 & 0 & \text{Im}(c_2) & \text{Im}(c_3) & 0 & 1 & \text{Re}(c_2) & \text{Re}(c_3) \\
0 & 0 & 0 & \text{Im}(c_4) & 0 & 0 & 1/2 & \text{Re}(c_4) \\
0 & 0 & 0 & 0 & 0 & 0 & 0 & 1/2
\end{pmatrix},
\]
where  $c_1,c_2,c_3,c_4$ are defined as:
\[
c_1 = \frac{\bar{g}\bar{h}}{2}, \quad c_2 = \frac{\bar{g}\bar{h}(1+i)}{2}, \quad c_3 = \frac{\bar{g}\bar{h}^2}{2}, \quad c_4 = \frac{|g|^2|h|^2 \bar{h}(1+i)}{2}.
\]

The Gram matrix $G_{8}^{(8)} = (\langle b_r, b_s \rangle)_{1\le r,s \le 8}$ associated with the ordered basis $\mathcal{B}_{\mathbb{Z}}$ is the block matrix
\[
G_{8}^{(8)} = \begin{pmatrix} M & N \\ -N & M \end{pmatrix},
\]
where $M$ is symmetric and $N$ is antisymmetric. The block entries are given by:
\[
M = \begin{pmatrix}
8 & 0 & 4\text{Re}(gh) & 0 \\
0 & 8|m|^{1/2} & 4|m|^{1/2}\text{Re}(gh(1-i)) & 4|m|^{1/2}\text{Re}(gh^2) \\
4\text{Re}(gh) & 4|m|^{1/2}\text{Re}(gh(1-i)) & \Omega_8 & \text{Re}(\Lambda_8) \\
0 & 4|m|^{1/2}\text{Re}(gh^2) & \text{Re}(\Lambda_8) & \Sigma_8
\end{pmatrix}
\]
and
\[
N = \begin{pmatrix}
0 & 0 & 4\text{Im}(gh) & 0 \\
0 & 0 & 4|m|^{1/2}\text{Im}(gh(1-i)) & 4|m|^{1/2}\text{Im}(gh^2) \\
-4\text{Im}(gh) & -4|m|^{1/2}\text{Im}(gh(1-i)) & 0 & \text{Im}(\Lambda_8) \\
0 & -4|m|^{1/2}\text{Im}(gh^2) & -\text{Im}(\Lambda_8) & 0
\end{pmatrix},
\]
where $\Omega_{8}$, $\Sigma_{8}$ and $\Lambda_{8}$ are given by:
\begin{align*}
\Omega_8 &= 2|g|^2|h|^2 (1 + 2|m|^{1/2}) + 2|f||h|, \\
\Sigma_8 &= 2|g|^2|h|^4 (|m|^{1/2} + 2|m|) + \frac{2|m|^{3/2}}{|g|^2|h|^4}, \\
\Lambda_8 &= 2|m|^{1/2}|g|^2|h|^2 h(1+i) + 2|m| h(1-i).
\end{align*}
\end{proposition}

\begin{proposition}
Let $L = \mathbb{Q}(i, \sqrt[4]{m})$ be an octic Kummer extension satisfying Case 9 conditions $(m \equiv 3 \pmod 4)$. Let $\mathcal{B}_{\mathbb{Z}} = \{u_1, \dots, u_4, \, i u_1, \dots, i u_4\}$ be the integral basis defined in Table 1.

The transition matrix $C_{9}$, which expresses the Case 9 basis $\mathcal{B}_{\mathbb{Z}}$ in terms of the Case 11 diagonal basis, is given by:

\[
C_{9} = \begin{pmatrix}
1 & 0 & \text{Re}(c_1) & \text{Re}(c_2) & 0 & 0 & -\text{Im}(c_1) & -\text{Im}(c_2) \\
0 & 1 & 0 & \text{Re}(c_2) & 0 & 0 & 0 & -\text{Im}(c_2) \\
0 & 0 & 1/2 & \text{Re}(c_3) & 0 & 0 & 0 & -\text{Im}(c_3) \\
0 & 0 & 0 & 1/2 & 0 & 0 & 0 & 0 \\
0 & 0 & \text{Im}(c_1) & \text{Im}(c_2) & 1 & 0 & \text{Re}(c_1) & \text{Re}(c_2) \\
0 & 0 & 0 & \text{Im}(c_2) & 0 & 1 & 0 & \text{Re}(c_2) \\
0 & 0 & 0 & \text{Im}(c_3) & 0 & 0 & 1/2 & \text{Re}(c_3) \\
0 & 0 & 0 & 0 & 0 & 0 & 0 & 1/2
\end{pmatrix},
\]
where $c_1,c_2,c_3$ are defined as:
\[
c_1 = \frac{i\bar{g}\bar{h}}{2}, \quad c_2 = \frac{i\bar{g}\bar{h}^2}{2}, \quad c_3 = \frac{|g|^2 h \bar{h}^2}{2}.
\]

The Gram matrix $G_{9}^{(8)} = (\langle b_r, b_s \rangle)_{1\le r,s \le 8}$ associated with the ordered basis $\mathcal{B}_{\mathbb{Z}}$ is the block matrix
\[
G_{9}^{(8)} = \begin{pmatrix} M & N \\ -N & M \end{pmatrix},
\]
where $M$ is symmetric and $N$ is antisymmetric. The block entries are given by:
\[
M = \begin{pmatrix}
8 & 0 & 4\text{Im}(gh) & 4\text{Im}(gh^2) \\
0 & 8|m|^{1/2} & 0 & 4|m|^{1/2}\text{Im}(gh^2) \\
4\text{Im}(gh) & 0 & 2|g|^2|h|^2 + 2|f||h| & \Omega_9\text{Re}(h) \\
4\text{Im}(gh^2) & 4|m|^{1/2}\text{Im}(gh^2) & \Omega_9\text{Re}(h) & \Sigma_9
\end{pmatrix}
\]
and
\[
N = \begin{pmatrix}
0 & 0 & -4\text{Re}(gh) & -4\text{Re}(gh^2) \\
0 & 0 & 0 & -4|m|^{1/2}\text{Re}(gh^2) \\
4\text{Re}(gh) & 0 & 0 & \Omega_9\text{Im}(h) \\
4\text{Re}(gh^2) & 4|m|^{1/2}\text{Re}(gh^2) & -\Omega_9\text{Im}(h) & 0
\end{pmatrix},
\]
where $\Omega_{9}$ and $\Sigma_{9}$ are given by:
\begin{align*}
\Omega_9 &= 2(|g|^2|h|^2 + |m|), \\
\Sigma_9 &= 2|g|^2|h|^4 (1 + |m|^{1/2} + |m|) + \frac{2|m|^{3/2}}{|g|^2|h|^4}.
\end{align*}
\end{proposition}

\begin{proposition}
Let $L = \mathbb{Q}(i, \sqrt[4]{m})$ be an octic Kummer extension satisfying Case 1 conditions $(m \equiv 1 \pmod 8)$. Let $\mathcal{B}_{\mathbb{Z}} = \{u_1, \dots, u_4, \, i u_1, \dots, i u_4\}$ be the integral basis defined in Table 1.

The transition matrix $C_{1}$, which expresses the Case 1 basis $\mathcal{B}_{\mathbb{Z}}$ in terms of the Case 11 diagonal basis, is given by:

\[
C_{1} = \begin{pmatrix}
1 & 1/2 & \text{Re}(c_1) & \text{Re}(c_3) & 0 & 1/2 & -\text{Im}(c_1) & -\text{Im}(c_3) \\
0 & 1/2 & \text{Re}(c_2) & \text{Re}(c_3) & 0 & 1/2 & -\text{Im}(c_2) & -\text{Im}(c_3) \\
0 & 0 & 1/4 & \text{Re}(c_4) & 0 & 0 & 1/4 & -\text{Im}(c_4) \\
0 & 0 & 0 & 1/4 & 0 & 0 & 0 & 0 \\
0 & -1/2 & \text{Im}(c_1) & \text{Im}(c_3) & 1 & 1/2 & \text{Re}(c_1) & \text{Re}(c_3) \\
0 & -1/2 & \text{Im}(c_2) & \text{Im}(c_3) & 0 & 1/2 & \text{Re}(c_2) & \text{Re}(c_3) \\
0 & 0 & -1/4 & \text{Im}(c_4) & 0 & 0 & 1/4 & \text{Re}(c_4) \\
0 & 0 & 0 & 0 & 0 & 0 & 0 & 1/4
\end{pmatrix},
\]
where $c_1,c_2,c_3$ and $c_4$ are defined as:
\[
c_1 = \frac{(1+i)\bar{g}\bar{h}}{4}, \quad c_2 = \frac{\bar{g}\bar{h}}{2}, \quad c_3 = \frac{\bar{g}\bar{h}^2}{4}, \quad c_4 = \frac{|g|^2 h \bar{h}^2}{4}.
\]

The Gram matrix $G_1^{(8)} = (\langle b_r, b_s \rangle)_{1\le r,s \le 8}$ associated with the ordered basis $\mathcal{B}_{\mathbb{Z}}$ is the block matrix
\[
G_1^{(8)} = \begin{pmatrix} M & N \\ -N & M \end{pmatrix},
\]
where $M$ is symmetric and $N$ is antisymmetric. The block entries are given by:
\[
M = \begin{pmatrix}
8 & 4 & 2\text{Re}(gh(1-i)) & 2\text{Re}(gh^2) \\
4 & 4(1+|m|^{1/2}) & \text{Re}(Z_1) & (1+|m|^{1/2})\text{Re}(gh^2(1-i)) \\
2\text{Re}(gh(1-i)) & \text{Re}(Z_1) & \Omega_1 & \text{Re}(\Lambda_1) \\
2\text{Re}(gh^2) & (1+|m|^{1/2})\text{Re}(gh^2(1-i)) & \text{Re}(\Lambda_1) & \Sigma_1
\end{pmatrix}
\]
and
\[
N = \begin{pmatrix}
0 & 4 & 2\text{Im}(gh(1-i)) & 2\text{Im}(gh^2) \\
-4 & 0 & \text{Im}(Z_1) & (1+|m|^{1/2})\text{Im}(gh^2(1-i)) \\
-2\text{Im}(gh(1-i)) & -\text{Im}(Z_1) & 0 & \text{Im}(\Lambda_1) \\
-2\text{Im}(gh^2) & -(1+|m|^{1/2})\text{Im}(gh^2(1-i)) & -\text{Im}(\Lambda_1) & 0
\end{pmatrix},
\]
where $\Omega_{1}$, $\Sigma_{1}$, $\Lambda_1$ and $Z_1$ are given by:
\begin{align*}
\Omega_1 &= |g|^2|h|^2(1 + 2|m|^{1/2}) + |f||h|, \\
\Sigma_1 &= \frac{1}{2}|g|^2|h|^4(1 + |m|^{1/2} + |m|) + \frac{|m|^{3/2}}{2|g|^2|h|^4}, \\
\Lambda_1 &= \frac{1}{2}|g|^2|h|^2 h(1+i) + |m|^{1/2}|g|^2|h|^2 h + \frac{1}{2}|m| h(1-i), \\
Z_1 &= -2igh + 2|m|^{1/2}gh(1-i).
\end{align*}
\end{proposition}

\begin{proposition}
Let $L = \mathbb{Q}(i, \sqrt[4]{m})$ be an octic Kummer extension satisfying Case 2 conditions $(m \equiv 5 \pmod{16})$. Let $\mathcal{B}_{\mathbb{Z}} = \{u_1, \dots, u_4, \, i u_1, \dots, i u_4\}$ be the integral basis defined in Table 1.

The transition matrix $C_{2}$, which expresses the Case 2 basis $\mathcal{B}_{\mathbb{Z}}$ in terms of the Case 11 diagonal basis, is given by:

\[
C_{2} = \begin{pmatrix}
1 & 1/2 & \text{Re}(c_1) & \text{Re}(c_2) & 0 & 1/2 & -\text{Im}(c_1) & -\text{Im}(c_2) \\
0 & 1/2 & \text{Re}(c_1) & \text{Re}(c_2) & 0 & 1/2 & -\text{Im}(c_1) & -\text{Im}(c_2) \\
0 & 0 & 1/4 & \text{Re}(c_3) & 0 & 0 & 1/4 & -\text{Im}(c_3) \\
0 & 0 & 0 & 1/4 & 0 & 0 & 0 & 0 \\
0 & -1/2 & \text{Im}(c_1) & \text{Im}(c_2) & 1 & 1/2 & \text{Re}(c_1) & \text{Re}(c_2) \\
0 & -1/2 & \text{Im}(c_1) & \text{Im}(c_2) & 0 & 1/2 & \text{Re}(c_1) & \text{Re}(c_2) \\
0 & 0 & -1/4 & \text{Im}(c_3) & 0 & 0 & 1/4 & \text{Re}(c_3) \\
0 & 0 & 0 & 0 & 0 & 0 & 0 & 1/4
\end{pmatrix},
\]
where $c_1,c_2,c_3$ are defined as:
\[
c_1 = \frac{\bar{g}\bar{h}(1-i)}{4}, \quad c_2 = \frac{\bar{g}\bar{h}^2}{4}, \quad c_3 = \frac{|g|^2 h \bar{h}^2}{4}.
\]

The Gram matrix $G_2^{(8)} = (\langle b_r, b_s \rangle)_{1\le r,s \le 8}$ associated with the ordered basis $\mathcal{B}_{\mathbb{Z}}$ is the block matrix
\[
G_2^{(8)} = \begin{pmatrix} M & N \\ -N & M \end{pmatrix},
\]
where $M$ is symmetric and $N$ is antisymmetric. The block entries are given by:
\[
M = \begin{pmatrix}
8 & 4 & 2\text{Re}(gh(1+i)) & 2\text{Re}(gh^2) \\
4 & 4(1+|m|^{1/2}) & \text{Re}(Z_2) & (1+|m|^{1/2})\text{Re}(gh^2(1-i)) \\
2\text{Re}(gh(1+i)) & \text{Re}(Z_2) & \Omega_2 & \text{Re}(\Lambda_2) \\
2\text{Re}(gh^2) & (1+|m|^{1/2})\text{Re}(gh^2(1-i)) & \text{Re}(\Lambda_2) & \Sigma_2
\end{pmatrix}
\]
and
\[
N = \begin{pmatrix}
0 & 4 & 2\text{Im}(gh(1+i)) & 2\text{Im}(gh^2) \\
-4 & 0 & \text{Im}(Z_2) & (1+|m|^{1/2})\text{Im}(gh^2(1-i)) \\
-2\text{Im}(gh(1+i)) & -\text{Im}(Z_2) & 0 & \text{Im}(\Lambda_2) \\
-2\text{Im}(gh^2) & -(1+|m|^{1/2})\text{Im}(gh^2(1-i)) & -\text{Im}(\Lambda_2) & 0
\end{pmatrix},
\]
where $\Omega_{2}$, $\Sigma_{2}$, $\Lambda_2$ and $Z_2$ are given by:
\begin{align*}
\Omega_2 &= |g|^2|h|^2(1 + |m|^{1/2}) + |f||h|, \\
\Sigma_2 &= \frac{1}{2}|g|^2|h|^4(1 + |m|^{1/2} + |m|) + \frac{|m|^{3/2}}{2|g|^2|h|^4}, \\
\Lambda_2 &= \left( \frac{1}{2}|g|^2|h|^2(1 + |m|^{1/2}) + \frac{1}{2}|m| \right) h(1-i), \\
Z_2 &= 2(1 + |m|^{1/2}) gh.
\end{align*}
\end{proposition}

\subsection{Asymptotic Stability of Shapes via Renormalization}

We analyze the distribution of shapes as the absolute discriminant approaches infinity ($|\Delta_L| \to \infty$), which corresponds to $|m|\to \infty$. We recall the fundamental shape parameters $\lambda_1 := |f|/|h|$ and $\lambda_2 := |g|$. The central result of our analysis is that, asymptotically, the geometry of the projected lattice $J(\mathcal{O}_L)^{\perp}$ depends solely on the pair $(\lambda_1, \lambda_2)$. As $|m|$ grows, the dominant terms of the Gram matrix entries normalize to functions of these parameters, while all other terms represent lower-order perturbations that vanish in the limit.

\subsubsection{Type I (Cases 3--6, 11--12)}
In these cases, the lattice scales uniformly. The entries of the Gram matrix $G_i^{(8)}$ are dominated by terms linear in $|m|^{1/2}$. To isolate the shape, we apply a uniform scalar normalization $N = |m|^{-1/2}$.

For example, in Case 12, the Gram matrix $G_{12}^{(8)}$ contains entries such as $8A$, $8B$, and $4A(1+\lambda)$ (where $\lambda = \lambda_1$). Upon normalization by $|m|^{1/2}$:
\begin{itemize}
    \item Terms linear in $A = |m|^{1/2}$ become constants or functions of $\lambda_1$ (e.g., $4A(1+\lambda)$ normalizes to $4(1+\lambda_1)$).
    \item The term $B = |f||h|$ appears on the diagonal. By invoking \text{Lemma \ref{lem:iden}}, we establish that $B = \frac{\sqrt{\lambda_1}}{\lambda_2}|m|^{1/2}$. Thus, this entry normalizes to $\frac{\sqrt{\lambda_1}}{\lambda_2}$.
    \item Constant terms (e.g., 8) vanish asymptotically.
\end{itemize}
Consequently, for large $|m|$, the shape is governed purely by the parameters $\lambda_1$ and $\lambda_2$. A similar analysis applies to all Type I cases.
\subsubsection{Type II: (Cases 1, 2, 7--10)}
These cases exhibit multi-scale behaviour, where basis vectors grow at differing rates. We employ the scaling matrix $S_m:= \text{diag}(|m|^{-k_i})$ to normalize each basis vector according to its specific growth exponent.

\par Consider Case 1 as an illustrative example. Using the identities established in \text{Lemma \ref{lem:iden}}, the diagonal entries exhibit the following asymptotics:
\begin{itemize}
    \item $G_{11} = 8$, \qquad $G_{22} = 4(1+|m|^{1/2}) \sim 4|m|^{1/2}$,
    \qquad $G_{33} = \Omega_1 \sim  \frac{2\lambda_2}{\sqrt{\lambda_1}}|m|$, \qquad $G_{44} = \Sigma_1 \sim  \frac{1}{2\lambda_1}|m|^2$.
\end{itemize}
Accordingly, we define the scaling matrix $S_m = \text{diag}\left(1, \, |m|^{-1/4}, \, |m|^{-1/2}, \, |m|^{-1}, \dots \right)$. The renormalized Gram matrix $S_m G_{1}^{(8)} S_m$ stabilizes such that its entries depend asymptotically only on $\lambda_1$ and $\lambda_2$. For instance, $G_{44}$ normalizes to  $\frac{1}{2\lambda_1}$. Proceeding similarly for the remaining entries, one can verify that the limiting shape is completely determined by the pair $(\lambda_1, \lambda_2)$. A similar analysis applies to all Type II cases.

\section{Density results}

\par
In this section, we study the distribution of the parameters
$\lambda_1$ and $\lambda_2$ from Definition~\ref{shape parameter defn}.
Assume that $(f,g,h)$ satisfy fixed congruence conditions modulo a high
power of $\mathfrak{l}:=(1+i)\Z[i]\subset \Z[i]$ so that we are in one of the cases from
Table~1.  For $\alpha\in\Z[i]$, set
\[
N(\alpha):=\op{Norm}_{K/\Q}(\alpha)=|\alpha|^2.
\]
\noindent There is a constant $C>0$ depending only on the congruence class,
the absolute discriminant equals
\[
|\Delta_L|=C N(f)^3N(g)^6N(h)^9
=C|fg^2h^3|^6.
\]
Accordingly, we normalize the height by
\[
H(f,g,h):=|f|\,|g|^2\,|h|^3.
\]

\par
Fix a rectangle
\[
\mathcal R:=[R_1',R_1]\times[R_2',R_2]\subset(0,\infty)^2
\]
and let $X>0$.
Let $\mathcal N'(\mathcal R;X)$ denote the set of triples $(f,g,h)$ of elements of $\Z[i]$ such that
\[
H(f,g,h)<X
\quad \text{and}\quad
(\lambda_1,\lambda_2)\in\mathcal R.
\]
We say that $(f,g,h)$ is \emph{strongly carefree} if $f$, $g$ and $h$ are
pairwise coprime and squarefree. Let $\mathcal N(\mathcal R;X)\subseteq \mathcal N'(\mathcal R;X)$ be the subset of strongly carefree $(f,g,h)$ and
write
\[
N(\mathcal R;X):=\#\mathcal N(\mathcal R;X)
\quad\text{and}\quad
N'(\mathcal R;X):=\#\mathcal N'(\mathcal R;X).
\]
\noindent Write $u:=|f|$, $v:=|g|$ and $w:=|h|$. The conditions $(\lambda_1,\lambda_2)\in\mathcal R$ become $\frac{u}{w}\in[R_1',R_1]$, $v\in[R_2',R_2]$ or equivalently $u\in[R_1'w,\;R_1 w]$. The height bound is $u\,v^2\,w^3<X$. Thus, after fixing $v$ and $w$, the admissible range for $u$ is
\begin{equation}\label{eq:u-range}
R_1'w
\;\le\;
u
\;\le\;
\min\!\left(
R_1w,\;
\frac{X}{v^2w^3}
\right).
\end{equation}

\subsection{Counting without arithmetic conditions}

\par
Before estimating $N(\mathcal R;X)$, we first estimate
$N'(\mathcal R;X)$, where no squarefreeness or coprimality conditions
are imposed.
It is easy to see that the number of Gaussian integers with $|z|\le T$
satisfies
\[
\#\{z\in\Z[i]:|z|\le T\}
=
\pi T^2+O(T).
\]
For fixed $v$ and $w$, the admissible range for $u$ is 
\begin{equation}\label{eq:u-range}
R_1'w
\;\le\;
u
\;\le\;
\min\!\left(
R_1 w,\;
\frac{X}{v^2w^3}
\right).
\end{equation}

\noindent For $t\geq 0$, let $N_t$ denote the number of $z\in \Z[i]$ such that $|z|=t$. Note that if $t^2\notin \Z$, then $N_t=0$. Using the Gaussian integer counting estimate in the $f$-variable, we obtain
\begin{equation}\label{N'(R;X) equation 1}
\begin{aligned}
N'(\mathcal R;X)
&=
\sum_{v}\sum_{w}
N_v N_w\left(
\pi\left(
\min\!\left(R_1w,\frac{X}{v^2w^3}\right)^2
-
(R_1'w)^2
\right)
+
O\!\left(
\min\!\left(R_1w,\frac{X}{v^2w^3}\right)
\right)
\right),
\end{aligned}
\end{equation}
where the sums range over $v,w>0$ with
\[
v\in[R_2',R_2],
\quad \text{and}\quad
w\le\left(\frac{X}{v^2R_1'}\right)^{1/4}.
\]
We note that since $v\ll 1$ and $w\ll X^{1/4}$, it follows that $N_v\ll 1$ and $N_w\ll X^\varepsilon$ for any $\varepsilon>0$.

\begin{lemma}
    We have the following bound for the error term:
\[
\sum_{v}\sum_{w} N_v N_w
\min\!\left(R_1w,\frac{X}{v^2w^3}\right)
\;\ll_{\mathcal R}\;
X^{1/2+\varepsilon}
\]
for any $\varepsilon >0$.
\end{lemma}
\begin{proof}
Fix $v\in[R_2',R_2]$.  The transition point at which
$R_1w=X/(v^2w^3)$ occurs at
\[
w_0=\left(\frac{X}{R_1v^2}\right)^{1/4}.
\]
We therefore split the $w$--sum at $w_0$:
\[
\sum_{w\le w_0} R_1w
\;+\;
\sum_{w>w_0}\frac{X}{v^2w^3}.
\]
\noindent The first of the above sums satisfies
\[
\sum_{w\le w_0} R_1w
\;\ll_{\mathcal R}
\frac{X^{1/2}}{v},
\]
while the second sum satisfies
\[
\sum_{w>w_0}\frac{X}{v^2w^3}
\;\ll\;
\frac{X}{v^2}\int_{w_0}^{\infty}w^{-3}\,dw
\;\ll \;
\frac{X}{v^2w_0^2}
\;\ll\;
\frac{X^{1/2}}{v}.
\]
\noindent Combining these bounds gives
\[
\sum_{w}
\min\!\left(R_1w,\frac{X}{v^2w^3}\right)
\;\ll_{\mathcal R}\;
\frac{X^{1/2}}{v}.
\]
Finally, summing over $v\in[R_2',R_2]$ and noting that $N_v\ll 1$ and $N_w\ll X^{\varepsilon}$, we have that
\[
\sum_{v}N_v \sum_{w} N_w
\min\!\left(R_1w,\frac{X}{v^2w^3}\right)
\;\ll_{\mathcal R}\;
X^{1/2+\varepsilon}.
\]
\end{proof}

\noindent Thus from \eqref{N'(R;X) equation 1}, we have that
\begin{equation}\label{eq:Nprime-sum-reduction}
\begin{aligned}
N'(\mathcal R;X)
&=
\sum_{v}N_v\sum_{w}
\pi N_w\left(
\min\!\left(R_1w,\frac{X}{v^2w^3}\right)^2
-
(R_1'w)^2
\right)
+
O_{\mathcal R}(X^{1/2+\varepsilon}).
\end{aligned}
\end{equation}

\begin{proposition}\label{Propn N' count}
As $X\to\infty$, one has
\[
N'(\mathcal R;X)
=
\pi^2\bigl(R_1-R_1'\bigr)
X
\sum_{v\in[R_2',R_2]}
\frac{N_v}{v^2}
+
O_{\mathcal R}(X^{3/4}).
\]
\end{proposition}

\begin{proof}
We begin from \eqref{eq:Nprime-sum-reduction}:
\[
N'(\mathcal R;X)
=
\pi
\sum_{v} N_v
\sum_{w}
N_w
\Big(
\min(R_1w,X/(v^2w^3))^2
-
(R_1'w)^2
\Big)
+
O_{\mathcal R}(X^{1/2+\varepsilon}),
\]
where $v\in[R_2',R_2]$ and the inner sum runs over $w$ for which the
summand is nonzero.

Fix $v$ in this interval.  Define
\[
F(t)
:=
\min(R_1t,X/(v^2t^3))^2
-
(R_1't)^2,
\]
and let
\[
T_v := \left(\frac{X}{R_1'v^2}\right)^{1/4}.
\]
Then the inner sum may be written as
\[
\sum_{w} N_w F(w)
=
\sum_{\substack{h\in\Z[i]\\ |h|\le T_v}} F(|h|).
\]
Setting $A(T):=\#\{h\in\Z[i]: |h|\le T\}$, one has that
\[
A(T)=\pi T^2+O(T).
\]
\noindent Write $0=r_0<r_1<\cdots<r_N\le T_v$ for the distinct values of $|h|$ in
increasing order.  Then
\[
\sum_{\substack{h\in\Z[i]\\ |h|\le T_v}} F(|h|)
=
\sum_{j=1}^N F(r_j)\bigl(A(r_j)-A(r_{j-1})\bigr).
\]
\noindent Applying Abel summation gives
\[
=
F(r_N)A(r_N)
-
\sum_{j=1}^{N-1} A(r_j)\bigl(F(r_{j+1})-F(r_j)\bigr).
\]
Since $F(t)\ll 1$ and $A(r_N)\ll T_v^2\asymp X^{1/2}$, we obtain
\[
F(r_N)A(r_N)=O(X^{1/2}).
\]
By the mean value theorem,
\[
F(r_{j+1})-F(r_j)
=
F'(\xi_j)(r_{j+1}-r_j)
\]
for some $\xi_j\in(r_j,r_{j+1})$.  Therefore
\[
\sum_{\substack{h\in\Z[i]\\ |h|\le T_v}} F(|h|)
=
-\sum_{j} A(r_j)F'(\xi_j)(r_{j+1}-r_j)
+
O(X^{1/2}).
\]
\noindent The sum on the right is a Riemann sum for the integral
\[
-\int_0^{T_v} A(r)F'(r)\,dr.
\]
Since $A(r)=\pi r^2+O(r)$, the difference between the discrete sum and the
integral is bounded by
\[
\ll \int_0^{T_v} r\,|F'(r)|\,dr.
\]
Hence
\[
\sum_{\substack{h\in\Z[i]\\ |h|\le T_v}} F(|h|)
=
-\int_0^{T_v} A(r)F'(r)\,dr
+
O\!\left(
\int_0^{T_v} r\,|F'(r)|\,dr
+
X^{1/2}
\right).
\]
\noindent Substituting $A(r)=\pi r^2+O(r)$ gives
\[
=
-\pi\int_0^{T_v} r^2F'(r)\,dr
+
O\!\left(
\int_0^{T_v} r\,|F'(r)|\,dr
+
X^{1/2}
\right).
\]
Note that $F(T_v)=F(0)=0$. Integration by parts yields
\[
-\pi\int_0^{T_v} r^2F'(r)\,dr
=
2\pi\int_0^{T_v} rF(r)\,dr.
\]

We may now evaluate these integrals. Let 
\[
r_0 := \left( \frac{X}{R_1 v^2} \right)^{1/4}
\quad\text{and}\quad
r_1 := \left( \frac{X}{R_1' v^2} \right)^{1/4}.
\]
By definition, the function $F$ is given by:
\[
F(r) = 
\begin{cases} 
(R_1^2 - R_1'^2) r^2, & 0 \le r \le r_0,\\[2mm]
\dfrac{X^2}{v^4 r^6} - R_1'^2 r^2, & r_0 < r \le r_1,\\[1mm]
0, & r > r_1.
\end{cases}
\]
First, consider the integral
\[
\int_0^{T_v} r F(r)\, dr 
= \int_0^{r_0} r (R_1^2 - R_1'^2) r^2\, dr 
+ \int_{r_0}^{r_1} r \left( \frac{X^2}{v^4 r^6} - R_1'^2 r^2 \right) dr.
\]
We have
\[
\int_0^{r_0} (R_1^2 - R_1'^2) r^3\, dr = \frac{R_1^2 - R_1'^2}{4} r_0^4 = \frac{R_1^2 - R_1'^2}{4} \cdot \frac{X}{R_1 v^2} = \frac{X (R_1 - R_1')(R_1 + R_1')}{4 R_1 v^2},
\]
and
\[
\int_{r_0}^{r_1} r \left( \frac{X^2}{v^4 r^6} - R_1'^2 r^2 \right) dr
= \frac{X^2}{v^4} \int_{r_0}^{r_1} r^{-5} dr - R_1'^2 \int_{r_0}^{r_1} r^3 dr.
\]

We compute each term separately:

\[
\frac{X^2}{v^4} \int_{r_0}^{r_1} r^{-5} dr = \frac{X^2}{v^4} \cdot \frac{r_0^{-4} - r_1^{-4}}{4} 
= \frac{X^2}{4 v^4} \left( \frac{R_1 v^2}{X} - \frac{R_1' v^2}{X} \right) 
= \frac{X (R_1 - R_1')}{4 v^2},
\]

and

\[
R_1'^2 \int_{r_0}^{r_1} r^3 dr = \frac{R_1'^2}{4} (r_1^4 - r_0^4) 
= \frac{R_1'^2}{4} \cdot \frac{X}{v^2} \left( \frac{1}{R_1'} - \frac{1}{R_1} \right) 
= \frac{X (R_1 - R_1') R_1'}{4 v^2 R_1}.
\]

Hence the second integral simplifies to
\[
\int_{r_0}^{r_1} r F(r)\, dr = \frac{X (R_1 - R_1')}{4 v^2} - \frac{X (R_1 - R_1') R_1'}{4 v^2 R_1} 
= \frac{X (R_1 - R_1')^2}{4 R_1 v^2}.
\]
\noindent Combining the two contributions, we obtain
\[
\int_0^{T_v} r F(r)\, dr = \frac{X (R_1 - R_1')(R_1 + R_1')}{4 R_1 v^2} + \frac{X (R_1 - R_1')^2}{4 R_1 v^2} = \frac{X (R_1 - R_1')}{2 v^2}.
\]
\noindent Multiplying by \(2\pi\) to account for angular directions gives
\[
2\pi \int_0^{T_v} r F(r)\, dr = \frac{\pi X (R_1 - R_1')}{v^2}.
\]

For the derivative, we have
\[
F'(r) = 
\begin{cases} 
2 (R_1^2 - R_1'^2) r, & 0 \le r \le r_0,\\
-6 \frac{X^2}{v^4 r^7} - 2 R_1'^2 r, & r_0 < r \le r_1.
\end{cases}
\]
Hence
\[
\int_0^{T_v} r |F'(r)| dr \ll \int_0^{r_0} r^2 dr + \int_{r_0}^{r_1} \left( \frac{X^2}{v^4 r^6} + r^2 \right) dr \ll \frac{X^{3/4}}{v^{3/2}}.
\]
\noindent Since \(v\) lies in a fixed interval \([R_2', R_2]\), this gives an error term
\[
O_{\mathcal R}(X^{3/4}).
\]
We conclude that
\[
\sum_{|h|\leq T_v} F(|h|) = \frac{\pi X (R_1 - R_1')}{v^2} + O_{\mathcal R}(X^{3/4}).
\]
Substitute this into the outer sum over $v$ with $\varepsilon\in (0, 1/4)$ to get:
\[
N'(\mathcal R;X)
=
\pi^2\bigl(R_1-R_1'\bigr)
X
\sum_{v\in[R_2',R_2]}
\frac{N_v}{v^2}
+
O_{\mathcal R}(X^{3/4}),
\]
which completes the proof.
\end{proof}

\subsection{Congruence conditions}
Given a nonzero prime ideal $\mathfrak p$ of $\Z[i]$, we say that a triple
$(\bar f,\bar g,\bar h)\in(\Z[i]/\mathfrak p^2)^3$ is
\emph{$\mathfrak p$-admissible} if
\[
\mathfrak p^2\nmid \bar f\bar g\bar h,
\qquad
\mathfrak p\nmid (\bar f,\bar g),\ (\bar f,\bar h),\ (\bar g,\bar h).
\]
Denote by $\mathcal A_{\mathfrak p}$ the set of all such admissible
residue classes. For a squarefree ideal $\mathfrak n\subset\Z[i]$, define
\[
\Omega_{\mathfrak n}
:=
\prod_{\mathfrak p\mid\mathfrak n}\mathcal A_{\mathfrak p}
\subset
(\Z[i]/\mathfrak n^2)^3,
\]
where the product is taken via the Chinese remainder theorem. 
\begin{lemma}\label{lem:local-density}
Let $\mathfrak p$ be a prime ideal of $\Z[i]$ with norm
$q=\op{Norm}(\mathfrak p)$. Let $\mathcal A_{\mathfrak p}\subset
(\Z[i]/\mathfrak p^2)^3$ denote the set of residue classes
$(f,g,h)$ such that
\begin{itemize}
\item no coordinate is divisible by $\mathfrak p^2$, and
\item at most one of $f,g,h$ is divisible by $\mathfrak p$.
\end{itemize}
Then
\[
\#\mathcal A_{\mathfrak p}
=
q^6\left(1-\frac{1}{q}\right)^2
\left(1+\frac{2}{q}-\frac{3}{q^2}\right).
\]
\end{lemma}

\begin{proof}
Since $\#(\Z[i]/\mathfrak p^2)=q^2$, the ambient space
$(\Z[i]/\mathfrak p^2)^3$ has cardinality $q^6$. We count the elements of
$\mathcal A_{\mathfrak p}$ by distinguishing cases according to the
$\mathfrak p$--adic valuations of the coordinates.
\par First suppose that $\mathfrak p$ divides none of $f,g,h$. The number of
elements of $\Z[i]/\mathfrak p^2$ not divisible by $\mathfrak p$ is
$q^2-q$. Hence the contribution from this case is $(q^2-q)^3$.
\par Next suppose that $\mathfrak p$ divides exactly one of the three
coordinates. By symmetry, it suffices to consider the case
$\mathfrak p\mid f$ and $\mathfrak p\nmid g,h$. Since divisibility by
$\mathfrak p^2$ is excluded, the element $f$ must be divisible by
$\mathfrak p$ but not by $\mathfrak p^2$. The number of such residue
classes modulo $\mathfrak p^2$ is $q-1$. Meanwhile, each of $g$ and $h$
must be coprime to $\mathfrak p$, giving $q^2-q$ choices for each.
Therefore, for a fixed choice of the coordinate divisible by
$\mathfrak p$, the number of admissible triples is $(q-1)(q^2-q)^2$. Since there are three choices for which coordinate is divisible by
$\mathfrak p$, the total contribution from this case is $3(q-1)(q^2-q)^2$.
\par Adding the contributions of the two admissible cases yields
\[
\#\mathcal A_{\mathfrak p}
=
(q^2-q)^3 + 3(q-1)(q^2-q)^2
=
(q^2-q)^2(q^2+2q-3).
\]
Factoring out $q^6$ gives the stated formula.
\end{proof}

Let $B>0$ and fix a squarefree ideal $\mathfrak n$ divisible by all prime ideals
$\mathfrak p$ with $\op{Norm}(\mathfrak p)\le B$. Let
\[
\mathcal C_{\mathfrak n}
:=
\left\{
(f,g,h)\in\Z[i]^3:
(f,g,h)\bmod\mathfrak n^2\in\Omega_{\mathfrak n}
\right\}.
\]
\noindent For a region $\mathcal R\subset\R_{>0}^2$ and
$X>0$, define
\[
N_{\mathfrak n}(\mathcal R;X)
:=
\#\left\{
(f,g,h)\in\mathcal C_{\mathfrak n}:
H(f,g,h)<X,\ (\lambda_1,\lambda_2)\in\mathcal R
\right\}.
\]

\begin{proposition}\label{prop:Nn-asymptotic}
Let $\mathcal R = [R_1',R_1]\times [R_2',R_2]$ be fixed intervals and $\mathfrak n \subset \Z[i]$ a nonzero ideal. Then, as $X \to \infty$, one has
\[
N_{\mathfrak n}(\mathcal R;X)
=
\pi^2
\bigl(R_1-R_1'\bigr)
X
\prod_{\mathfrak p\mid\mathfrak n}
\left(1-\dfrac{3}{\op{Norm}(\mathfrak p)^2}+\dfrac{2}{\op{Norm}(\mathfrak p)^3}\right)
\sum_{\substack{g\in\Z[i]\\
g\text{ squarefree}\\
|g|\in[R_2',R_2]}}
\frac{1}{|g|^2}
\prod_{\mathfrak p\mid g}
\left(
\frac{\op{Norm}(\mathfrak p)}
{\op{Norm}(\mathfrak p)+2}
\right)
+
O_{\mathcal R,\mathfrak n}(X^{3/4}).
\]
\end{proposition}

\begin{proof}
Let $\widetilde{\Omega}_{\mathfrak n} \subset (\Z[i]/\mathfrak n^2)^3$ be a fixed set of representatives for the congruence conditions. For $a = (a_1,a_2,a_3) \in \widetilde{\Omega}_{\mathfrak n}$, write
\[
f = a_1 + \mathfrak n^2 b_1, \quad
g = a_2 + \mathfrak n^2 b_2, \quad
h = a_3 + \mathfrak n^2 b_3, \qquad (b_1,b_2,b_3)\in \Z[i]^3.
\]
\noindent The shape conditions $(\lambda_1,\lambda_2) \in \mathcal R$ translate to
\[
R_1' |h| \le |f| \le R_1 |h|, \qquad R_2' \le |g| \le R_2,
\]
and the height constraint $H(f,g,h) = |f||g|^2 |h|^3 < X$ becomes
\[
|f||h|^3 < \frac{X}{|g|^2}.
\] 
Fix $g \in \Z[i]$ satisfying $|g|\in [R_2',R_2]$. Write $w := |h|$. Then $f$ must satisfy
\[
R_1' w \le |f| \le R(w), \qquad R(w) := \min\Big(R_1 w, X/(|g|^2 w^3)\Big),
\]
and the inequality $R(w) \ge R_1' w$ restricts $w$ to
\[
0 \le w \le W := \left(\frac{X}{R_1' |g|^2}\right)^{1/4}.
\]
The variables $(f,h)$ lie in a translate of the lattice
\[
\Lambda := \mathfrak n^2 \Z[i] \times \mathfrak n^2 \Z[i] \subset \R^4,
\]
with covolume $|\mathfrak n|^8$. Define the regions
\[
\begin{split}&\mathcal A_w := \{(f,h) \in \C^2: |h| = w, \; R_1' w \le |f| \le R(w)\},\\
& \mathcal A:=\{(f,h) \in \C^2: |h|\leq W, \; R_1' |h| \le |f| \le R(|h|)\}.
\end{split}
\] 
We compute the 4-dimensional volume of $\mathcal A$. For fixed $w$, the $h$-circle of radius $w$ has length $2\pi w$; the $f$-annulus has area $\pi(R(w)^2 - (R_1'w)^2)$. Hence
\[
\operatorname{Vol}_4(\mathcal A) = \int_0^W 2 \pi w \cdot \pi (R(w)^2 - R_1'^2 w^2) \, dw 
= 2 \pi^2 \int_0^W w \left(R(w)^2 - R_1'^2 w^2\right) \, dw.
\]
Define $w_0$ by
\[
R_1 w_0 = \frac{X}{|g|^2 w_0^3}\] i.e., $w_0 = \left(\frac{X}{R_1 |g|^2}\right)^{1/4}$. Then
\[
R(w) =
\begin{cases}
R_1 w, & 0 \le w \le w_0,\\
X/(|g|^2 w^3), & w_0 \le w \le W.
\end{cases}
\]
Split the integral as follows
\[
\operatorname{Vol}_4(\mathcal A) = 2 \pi^2 \left[ \int_0^{w_0} w \bigl( (R_1 w)^2 - (R_1' w)^2\bigr)\, dw + \int_{w_0}^{W} w \left( (X/(|g|^2 w^3))^2 - (R_1' w)^2 \right) dw \right].
\]
For the first integral,
\[
\int_0^{w_0} w \bigl( (R_1 w)^2 - (R_1' w)^2\bigr) dw 
= (R_1^2 - R_1'^2) \int_0^{w_0} w^3 dw = \frac{R_1^2 - R_1'^2}{4} w_0^4.
\]
Substitute $w_0^4 = X/(R_1 |g|^2)$:
\[
\int_0^{w_0} w \bigl( (R_1 w)^2 - (R_1' w)^2\bigr) dw = \frac{R_1^2 - R_1'^2}{4} \cdot \frac{X}{R_1 |g|^2}.
\]
For the second integral, compute each term explicitly:
\[
\frac{X^2}{|g|^4} \int_{w_0}^{W} w^{-5} dw - R_1'^2 \int_{w_0}^{W} w^3 dw
= \frac{X^2}{4 |g|^4} \Bigl( \frac{1}{w_0^4} - \frac{1}{W^4} \Bigr) - \frac{R_1'^2}{4} (W^4 - w_0^4).
\]
Substituting $w_0^4 = X/(R_1 |g|^2)$ and $W^4 = X/(R_1' |g|^2)$ gives
\[
\frac{X^2}{4 |g|^4 w_0^4} - \frac{X^2}{4 |g|^4 W^4} - \frac{R_1'^2}{4} W^4 + \frac{R_1'^2}{4} w_0^4
= \frac{X}{4 |g|^2} \Bigl( R_1 - 2 R_1' + \frac{R_1'^2}{R_1} \Bigr).
\] 
Adding the contribution from $[0,w_0]$, which equals $\frac{X}{4 |g|^2} \left(R_1 - \frac{R_1'^2}{R_1}\right)$, we obtain
\[
\int_0^W w \bigl(R(w)^2 - (R_1' w)^2\bigr) dw
= \frac{X}{4 |g|^2} \Bigl(R_1 - \frac{R_1'^2}{R_1} + R_1 - 2 R_1' + \frac{R_1'^2}{R_1} \Bigr)
= \frac{X}{2 |g|^2} (R_1 - R_1').
\]
Hence, the 4-dimensional volume of $\mathcal A$ is
\[
\operatorname{Vol}_4(\mathcal A) = 2 \pi^2 \int_0^W w \bigl(R(w)^2 - (R_1' w)^2\bigr) dw
= \frac{\pi^2 X}{|g|^2} (R_1 - R_1').
\] 
By Davenport's lemma in dimension $4$ \cite[Theorem on p.~180]{Davenport}, for the lattice $\Lambda \subset \R^4$ with covolume $|\mathfrak n|^8$, we have
\[
\#(\mathcal A \cap \Lambda) = \frac{\operatorname{Vol}_4(\mathcal A)}{|\mathfrak n|^8} + O\Bigg(\sum_{k=1}^{3} \sum_{\substack{I\subset \{1,2,3,4\} \\ |I|=k}} \operatorname{Vol}_k(\pi_I(\mathcal A)) \Bigg),
\]
where $\pi_I$ denotes the orthogonal projection onto the $k$-dimensional coordinate subspace. Each projection volume can be bounded as follows.
Let $(x_1,x_2,y_1,y_2)$ be the coordinates corresponding to $f = x_1+ix_2$, $h = y_1+iy_2$. Then we have that $|x_i|\leq W$ and $|y_i|\leq R_1 W$. Therefore, we find that 
\[\sum_{\substack{I\subset \{1,2,3,4\} \\ |I|=k}} \operatorname{Vol}_k(\pi_I(\mathcal A))\ll X^{k/4}.\]
Thus we find that
\[
\#\{(b_1,b_3): f = a_1 + \mathfrak n^2 b_1, h = a_3 + \mathfrak n^2 b_3, (f,h) \in \mathcal A\} = \frac{\pi^2 (R_1 - R_1') X}{ |g|^2 |\mathfrak n|^8} + O_{\mathcal{R},\mathfrak n}(X^{3/4}).
\]
\par Recalling that $g=a_2+\mathfrak n^2 b_2$, this yields
\[
N_{\mathfrak n}(\mathcal R;X)
=
\pi^2
\bigl(R_1-R_1'\bigr)
X\sum_{\substack{g\in\Z[i]\\ |g|\in[R_2',R_2]}} \frac{\#\{a\in\widetilde{\Omega}_{\mathfrak n}:\ a_2\equiv g \bmod \mathfrak n^2\}}{|\mathfrak{n}|^8|g|^{2}}
+
O_{\mathcal R, \mathfrak n}(X^{3/4}).
\]
We have that
\[\#\{a\in\widetilde{\Omega}_{\mathfrak n}:\ a_2\equiv g \bmod \mathfrak n^2\}=\prod_{\mathfrak p|\mathfrak n} \#\{a\in\widetilde{\Omega}_{\mathfrak n}:\ a_2\equiv g \bmod \mathfrak p^2\}\]
The cardinality of the fiber on the right-hand side factorizes over the
prime ideals dividing\/ $\mathfrak n$. Let $\mathfrak p$ be such a prime.
A direct local computation shows that
\[
\#\{a\in\widetilde{\Omega}_{\mathfrak p}:\ a_2\equiv g \bmod \mathfrak p^2\}
=
\begin{cases}
0, & \mathfrak p^2\mid g,\\[4pt]
\op{Norm}(\mathfrak p)^4
\left(1-\dfrac{3}{\op{Norm}(\mathfrak p)^2}+\dfrac{2}{\op{Norm}(\mathfrak p)^3}\right),
& \mathfrak p\nmid g,\\[8pt]
\op{Norm}(\mathfrak p)^4
\left(1-\dfrac{1}{\op{Norm}(\mathfrak p)}\right)^2,
& \mathfrak p\mid g,\ \mathfrak p^2\nmid g.
\end{cases}
\]
Multiplying over all $\mathfrak p\mid\mathfrak n$ yields
\[
\#\{a\in\widetilde{\Omega}_{\mathfrak n}:\ a_2\equiv g \bmod \mathfrak n^2\}
=
|\mathfrak{n}|^8
\prod_{\mathfrak p\mid\mathfrak n}
\left(1-\dfrac{3}{\op{Norm}(\mathfrak p)^2}+\dfrac{2}{\op{Norm}(\mathfrak p)^3}\right)
\prod_{\mathfrak p\mid g}
\left(
\frac{\op{Norm}(\mathfrak p)}
{\op{Norm}(\mathfrak p)+2}
\right).
\]
Substituting this expression into the previous formula completes the
proof.
\end{proof}

\begin{lemma}\label{lem:large-prime-tail}
Let $Y>0$ and let $\mathcal E_Y(\mathcal R;X)$ denote the set of triples
$(f,g,h)\in\Z[i]^3$ satisfying the height and shape conditions defining
$N(\mathcal R;X)$ but violating the strongly carefree condition at some
prime ideal $\mathfrak p$ with $\op{Norm}(\mathfrak p)>Y$. Then
\[
\frac{\#\mathcal E_Y(\mathcal R;X)}{X}
\ll_{\mathcal R}
\sum_{\op{Norm}(\mathfrak p)>Y}\op{Norm}(\mathfrak p)^{-2}.
\]
\end{lemma}

\begin{proof}
Since $\#N(\mathcal R;X)\ll_{\mathcal R} X$, it suffices to obtain
uniform bounds for primes $\mathfrak p$ of sufficiently large norm.
Fix such a prime ideal $\mathfrak p$. If $(f,g,h)\in\mathcal
E_Y(\mathcal R;X)$ violates the strongly carefree condition at
$\mathfrak p$, then at least one of the following occurs: $\mathfrak
p^2\mid f$, $\mathfrak p^2\mid g$, $\mathfrak p^2\mid h$, or $\mathfrak
p\mid(f,g)$, $\mathfrak p\mid(f,h)$, or $\mathfrak p\mid(g,h)$.

Consider first the case $\mathfrak p^2\mid f$. Writing $f=\mathfrak
p^2f'$, the height condition
\[
|f||g|^2|h|^3<X
\]
implies
\[
|f'||g|^2|h|^3<\frac{X}{\op{Norm}(\mathfrak p)^2}.
\]
Applying the geometry-of-numbers estimate Proposition \ref{Propn N' count} to this
rescaled region, with the same shape constraints as in $N(\mathcal
R;X)$, yields a bound
\[
\#\{(f,g,h):\mathfrak p^2\mid f\}
\ll_{\mathcal R}
\frac{X}{\op{Norm}(\mathfrak p)^{2}}.
\]
The remaining cases are
treated identically. Summing these bounds over all primes
$\mathfrak p$ with $\op{Norm}(\mathfrak p)>Y$ yields the stated result.
\end{proof}

\begin{theorem}\label{thm:main-density}
As $X\to\infty$, one has
\[
N(\mathcal R;X)
=
\pi^2
\bigl(R_1-R_1'\bigr)
X
\sum_{\substack{g\in\Z[i]\\
g\ \mathrm{squarefree}\\
|g|\in[R_2',R_2]}}
\frac{1}{|g|^{2}}
\prod_{\mathfrak p\mid g}
\left(
\frac{\op{Norm}(\mathfrak p)}
{\op{Norm}(\mathfrak p)+2}
\right)
\prod_{\mathfrak p}
\left(1-\dfrac{3}{\op{Norm}(\mathfrak p)^2}+\dfrac{2}{\op{Norm}(\mathfrak p)^3}\right)
+
O_{\mathcal R}(X^{3/4}).
\]
\end{theorem}

\begin{proof}
Let $\mathfrak n_Y$ denote the product of all prime ideals of $\Z[i]$
with norm at most $Y$. By definition of the strongly carefree condition,
we have
\[
N_{\mathfrak n_Y}(\mathcal R;X)
\le
N(\mathcal R;X)
\le
N_{\mathfrak n_Y}(\mathcal R;X)
+
\#\mathcal E_Y(\mathcal R;X).
\]
Dividing by $X^{3/4}$ and letting $X \to\infty$, Proposition~\ref{prop:Nn-asymptotic}
implies that
\[
\lim_{X\to\infty}
\frac{N_{\mathfrak n_Y}(\mathcal R;X)}{X}
=
\pi^2
\bigl(R_1-R_1'\bigr)
\prod_{\mathfrak p\mid\mathfrak n_Y}
\left(1-\dfrac{3}{\op{Norm}(\mathfrak p)^2}+\dfrac{2}{\op{Norm}(\mathfrak p)^3}\right)
\sum_{\substack{g\in\Z[i]\\
g\ \mathrm{squarefree}\\
|g|\in[R_2',R_2]}}
\frac{1}{|g|^{2}}
\prod_{\mathfrak p\mid g}
\left(
\frac{\op{Norm}(\mathfrak p)}
{\op{Norm}(\mathfrak p)+2}
\right).
\]
By Lemma~\ref{lem:large-prime-tail}, the contribution of
$\mathcal E_Y(\mathcal R;X)$ becomes negligible after dividing by
$X$ and letting $Y\to\infty$. Since
\[
\prod_{\mathfrak p\mid\mathfrak n_Y}
\left(1-\dfrac{3}{\op{Norm}(\mathfrak p)^2}+\dfrac{2}{\op{Norm}(\mathfrak p)^3}\right)\] converges to
\[\prod_{\mathfrak p}
\left(1-\dfrac{3}{\op{Norm}(\mathfrak p)^2}+\dfrac{2}{\op{Norm}(\mathfrak p)^3}\right),
\]
the claim follows.
\end{proof}

\noindent We now give a proof of our main distribution result.
\begin{proof}[Proof of Theorem \ref{thm a}]We find that \[\int_{\mathcal R}\hat{\mu}=\pi^2
\bigl(R_1-R_1'\bigr)
\sum_{\substack{g\in\Z[i]\\
g\ \mathrm{squarefree}\\
|g|\in[R_2',R_2]}}
\frac{1}{|g|^{2}}
\prod_{\mathfrak p\mid g}
\left(
\frac{\op{Norm}(\mathfrak p)}
{\op{Norm}(\mathfrak p)+2}
\right)
\prod_{\mathfrak p}
\left(1-\dfrac{3}{\op{Norm}(\mathfrak p)^2}+\dfrac{2}{\op{Norm}(\mathfrak p)^3}\right).\]
    Thus the result follows directly from Theorem \ref{thm:main-density}. 
\end{proof}
\bibliographystyle{alpha}
\bibliography{references}

\end{document}